\documentclass[12pt,oneside,reqno]{amsart}

\usepackage{amsmath,amssymb,amsthm,textcomp}
\usepackage{dsfont}
\usepackage{amsfonts,graphicx}
\usepackage[mathscr]{eucal}
\usepackage{color}
\usepackage{csquotes}
\usepackage{relsize}
\usepackage{hyperref}
\usepackage[backend=bibtex,%
firstinits=true,%
doi=false,%
isbn=true,%
url=false,%
maxnames=99]{biblatex}%

\usepackage{pgfplots}
\pgfplotsset{compat=1.18}

\AtEveryBibitem{\clearfield{issn}}
\AtEveryCitekey{\clearfield{issn}}
\numberwithin{equation}{section}
\DeclareNameAlias{sortname}{last-first}
\theoremstyle{definition}
\usepackage{mathtools}
\numberwithin{equation}{section}

\newcommand{\ncom}{\newcommand}

\ncom{\beq}{\begin{equation}}
	\ncom{\eeq}{\end{equation}}
\ncom{\bea}{\begin{eqnarray*}}
	\ncom{\eea}{\end{eqnarray*}}
\ncom{\beqa}{\begin{eqnarray}}
	\ncom{\eeqa}{\end{eqnarray}}
\ncom{\nno}{\nonumber}
\ncom{\non}{\nonumber}
\ncom{\ds}{\displaystyle}
\ncom{\half}{\frac{1}{2}}
\ncom{\mbx}{\makebox{.25cm}}
\ncom{\hs}{\mbox{\hspace{.25cm}}}
\ncom{\rar}{\rightarrow}
\ncom{\Rar}{\Rightarrow}
\ncom{\noin}{\noindent}
\ncom{\bc}{\begin{center}}
	\ncom{\ec}{\end{center}}
\ncom{\sz}{\scriptsize}
\ncom{\rf}{\ref}
\ncom{\s}{\sqrt{2}}
\ncom{\sgm}{\sigma}
\ncom{\Sgm}{\Sigma}
\ncom{\psgm}{\sigma^{\prime}}
\ncom{\dt}{\delta}
\ncom{\Dt}{\Delta}
\ncom{\lmd}{\lambda}
\ncom{\Lmd}{\Lambda}
\ncom{\Th}{\Theta}
\ncom{\e}{\eta}
\ncom{\eps}{\epsilon}
\ncom{\pcc}{\stackrel{P}{>}}
\ncom{\lp}{\stackrel{L_{p}}{>}}
\ncom{\dist}{{\rm\,dist}}
\ncom{\sspan}{{\rm\,span}}
\ncom{\re}{{\rm Re\,}}
\ncom{\im}{{\rm Im\,}}
\ncom{\sgn}{{\rm sgn\,}}
\ncom{\ba}{\begin{array}}
	\ncom{\ea}{\end{array}}
\ncom{\hone}{\mbox{\hspace{1em}}}
\ncom{\htwo}{\mbox{\hspace{2em}}}
\ncom{\hthree}{\mbox{\hspace{3em}}}
\ncom{\hfour}{\mbox{\hspace{4em}}}
\ncom{\vone}{\vskip 2ex}
\ncom{\vtwo}{\vskip 4ex}
\ncom{\vonee}{\vskip 1.5ex}
\ncom{\vthree}{\vskip 6ex}
\ncom{\vfour}{\vspace*{8ex}}
\ncom{\norm}{\|\;\;\|}
\ncom{\integ}[4]{\int_{#1}^{#2}\,{#3}\,d{#4}}
\ncom{\vspan}[1]{{{\rm\,span}\{ #1 \}}}
\ncom{\dm}[1]{ {\displaystyle{#1} } }
\ncom{\ri}[1]{{#1} \index{#1}}

\newtheorem{theorem}{\bf Theorem}[section]
\newtheorem{remark}{\bf Remark}[section]
\newtheorem{proposition}{Proposition}[section]
\newtheorem{lemma}{Lemma}[section]

\newtheoremstyle
{remarkstyle}
{}
{11pt}
{}
{}
{\bfseries}
{:}
{     }
{\thmname{#1} \thmnumber{#2} }

\theoremstyle{remarkstyle}

\def\eps{\varepsilon}

\begin{document}
	\title{On a Telegraph Process with Generalized Mittag-Leffler Waiting times and Velocity Driven by Random trials}
		\author[Rohini Bhagwanrao Pote]{Rohini Bhagwanrao Pote}
		\address{Rohini Bhagwanrao pote, Department of Mathematics, Indian Institute of Technology Bhilai, Durg 491002, India.}
		\email{rohinib@iitbhilai.ac.in}
		\author[Kuldeep Kumar Kataria]{Kuldeep Kumar Kataria}
		\address{Kuldeep Kumar Kataria, Department of Mathematics, Indian Institute of Technology Bhilai, Durg 491002, India.}
		\email{kuldeepk@iitbhilai.ac.in}
	\date{\today}
	
	\begin{abstract}
		We study a generalized telegraph process in which the velocity is governed by random trials by considering the specific distribution of waiting times. In this telegraph process, a particle moving on real line may change its direction whenever there is an arrival in a counting process. This direction change is driven by the outcomes of random trials. In the first case, random trials are independent and identically distributed, and waiting times have Mittag-Leffler distribution. In the second case, random trials follow {\it P\'olya urn scheme} and the first waiting time is generalized Mittag-Leffler distributed whereas other waiting times have Mittag-Leffler distribution. In both cases, we obtain the discrete component of their probability law. Also, absolutely continuous components of their conditional probability law given initial velocity are derived. The plots of absolutely continuous components of their probability law are compared for different parameters. Conditional on the initial velocity, the distributions of $n$th event time of counting processes associated with these telegraph processes are obtained.
	\end{abstract}
	\maketitle 
	
	\section{Introduction}\label{SecIntro}
	A continuous time stochastic process that models motion of a particle moving on real line with constant speed whose direction is reversed at Poissonian events is known as the telegraph process. Its transition density satisfies a second-order hyperbolic telegraph equation (see Kac (1974)). In the last few decades, many authors have studied the telegraph process and its generalizations, for example, motions with reflecting and absorbing barriers driven by the telegraph equation (see Orsingher (1995)), properties of the telegrapher's random process with or without a trap (see Foong and Kanno (1994)), telegraph process with velocities alternating at Erlang-distributed random times (see Di Crescenzo (2001)), telegraph process with gamma-distributed alternating velocities in biological modeling (see Di Crescenzo and Martinucci (2007)), telegraph random evolution on a circle (see De Gregorio and Iafrate (2021)), {\it etc}. Recently, Di Crescenzo {\it et al.} (2026) studied a generalized telegraph process with resetting to the origin driven by Bernoulli trials.
	
	Several fractional versions of the telegraph process are introduced and studied in literature. These are obtained by replacing the standard derivative in telegraph equation by a suitable fractional derivative with respect to space or time variable. Some related works in this direction are as follows: the time-fractional telegraph equations and telegraph processes with Brownian time (see Orsingher and Beghin (2004)), the space-fractional telegraph equation and related fractional telegraph process (see Orsingher and Zhao (2003)), time-fractional telegraph equation with $\psi$-Hilfer derivatives (see Vieira {\it et al.} (2022)), {\it etc}. Few other generalizations of telegraph process are as follows: a random telegraph signal of Mittag-Leffler type (see Ferraro {\it et al.} (2009)), the telegraph process stopped at stable-distributed times and its connection with fractional telegraph equation (see Beghin and Orsingher (2003)), {\it etc}. Masoliver (2016) relates the fractional telegraph equation to anomalous transport processes with non-exponential waiting time distribution. Study of these processes involves fractional derivatives and special funcions such as Caputo fractional derivative, Mittag-Leffler function, {\it etc}. 
	Such processes are of non-Markovian nature and exhibits long-range behavior. 
	
	A telegraph process in which the particle's velocity is chosen randomly at Poissonian events is studied by Stadje and Zacks (2004). Di Crescenzo and Martinucci (2013) considered the generalized telegraph process with deterministic jumps. Later, this study is extended by considering the velocity changes driven by alternating fractional Poisson process (see Di Crescenzo and Meoli (2018)). A generalized telegraph process with velocity driven by random trials is introduced and studied by Crimaldi {\it et al.} (2013). In this telegraph process, the particle moves at constant speed and whenever there is an occurrence of an event in counting process it may change its direction. This direction change is governed by random trials which are divided into two cases, namely, the {\it Bernoulli scheme} in which random trials are Bernoulli independent and identically distributed (iid) and the {\it P\'olya urn scheme} where outcome of the next trial is based on previous outcomes. Further, Crimaldi {\it et al.} (2013) studied this model for exponentially distributed and gamma distributed waiting times. For the case of exponentially distributed waiting times, Macci (2016) proved a large deviation principle. 
	
	\subsection{Generalized telegraph process with velocity driven by random trials}\label{SecStochModel}
	Crimaldi {\it et al.} (2013) introduced a generalized telegraph process with velocity driven by random trials. It is a telegraph process in which the particle moves on the real line with one of the two possible constant velocities that depends on its direction. A change in particle's direction is governed by the outcomes of random trials which are performed at event times of a counting process.
	
	Mathematically, this telegraph process is described as follows:
	
	Let $X(t)$ be the position of a particle moving on the real line and $V(t)$ be its velocity at time $t\geq0$. Let $v_1$ and $v_2$ be two positive constants such that the particle moves with velocity $v_{1}$ in the positive direction and with velocity $-v_{2}$ in the negative direction. Let $T_n$ be the $n$th event time of a counting process $\{\mathcal{N}(t)\}_{t\geq0}$. Then, $\mathcal{N}(t)\coloneqq\max\{n\in\mathbb{N}_{0}:T_n \leq t\}$, $t\geq0$. It is assumed that the particle starts moving at time $T_0 =0$ from the  origin, that is, $X(0)=0$. At time $T_0 =0$, a random trial is performed whose outcome $Y_1$ determines the initial velocity of the particle. At the first event time $T_1$, the second random trial is performed whose outcome $Y_2$ determines the future velocity of the particle starting from $T_1$. So, for $n\in\mathbb{N}_{0}$, at event time $T_n$, the $(n+1)$th random trial is performed whose outcome $Y_{n+1}$ decides the velocity of the particle till the next event time $T_{n+1}$. 
	The duration of the particle's $n$th interval of motion in the positive and negative directions are denoted by $R_n$ and $L_n$, respectively. Let $S_n$ be the velocity of the particle during the time interval $[T_n, T_{n+1})$ which is determined by the sequence $\{Y_n\}_{n\in\mathbb{N}}$ of random trials. Here, $\{R_n\}_{n\in\mathbb{N}}$, $\{L_n\}_{n\in\mathbb{N}}$ and $\{Y_n\}_{n\in\mathbb{N}}$ are assumed to be mutually independent. 
	
	\begin{figure}
		\centering
		\includegraphics[width=1\textwidth]{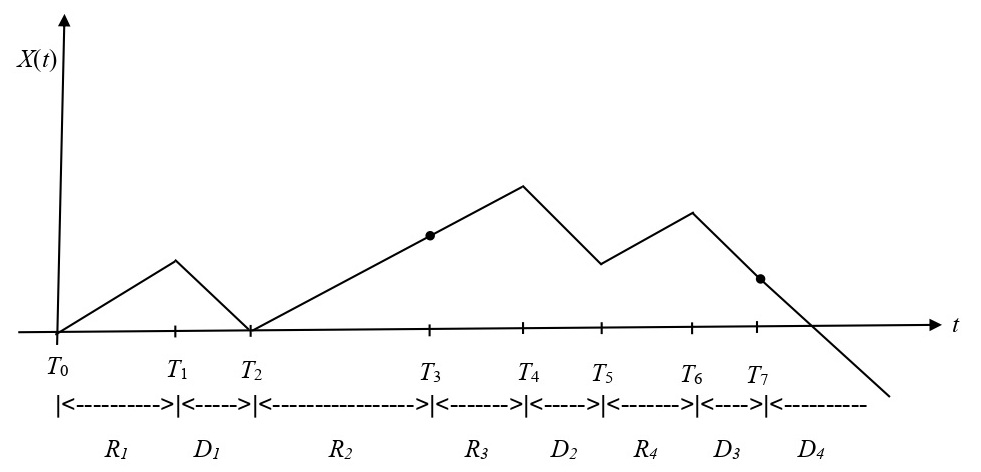}
		\caption{Sample paths of $X(t)$ for $V_{0}=v_1$.}
		\label{fig1} 
	\end{figure}
	
	Let $\mathscr{F}_0 =\{\phi, \Omega\}$ and $\mathscr{F}_m=\sigma(Y_1, Y_2, \dots, Y_m)$, $m\in\mathbb{N}$. The sequence of random trials $\{Y_n\}_{n\in\mathbb{N}}$ is described as follows:
	\begin{align*}
		\mathrm{Pr}\{S_0 =v_1|\mathscr{F}_{0}\}&=\mathrm{Pr}\{Y_1 =1|\mathscr{F}_{0}\}=\frac{b}{b+r},\\\nonumber
		\mathrm{Pr}\{S_0 =-v_2|\mathscr{F}_{0}\}&=\mathrm{Pr}\{Y_1 =0|\mathscr{F}_{0}\}=\frac{r}{b+r},
	\end{align*}
	and, for $m\in\mathbb{N}$,
	\begin{align*}
		\mathrm{Pr}\{S_m =v_1|\mathscr{F}_m\}&=\mathrm{Pr}\{Y_{m+1} =1|\mathscr{F}_m\}=\frac{b+c\sum_{i=1}^{m}Y_i}{b+r+cm},\\\nonumber
		\mathrm{Pr}\{S_{m} =-v_2|\mathscr{F}_m\}&=\mathrm{Pr}\{Y_{m+1} =0|\mathscr{F}_m\}=\frac{r+c\sum_{i=1}^{m}(1-Y_i)}{b+r+cm},
	\end{align*}
	where $b$ and $r$ are positive constants and $c$ is a non-negative constant. 
	
	The following two cases are considered:
	
	Case I. For $c=0$, $\{Y_n\}_{n\in\mathbb{N}}$ is a sequence of independent Bernoulli distributed random variables with success probability $p={b}/{(b+r)}$, that is, $Y_n$'s are iid. This case is referred as {\it Bernoulli scheme}.
	
	Case II. For $c>0$, $Y_n$'s are no longer iid. In case $b$, $r$ and $c$ are integer valued, $\{Y_n\}_{n\in\mathbb{N}}$ can be viewed as the {\it classical P\'olya urn scheme}.
	
	The particle's position at time $t$ is given by $X(t)=\int_{0}^{t}V(s)\mathrm{d}s$, where $V(t)=S_{\mathcal{N}(t)}$ is its velocity.
	
	We denote $M_{n-1}$ as the random variable that counts the number of random trials determining the velocity $v_1$ starting from the second till the $n$th trial. Thus, we have
	\begin{equation*}
		M_{n-1}=\sum_{i=1}^{n-1}\mathds{1}_{\{S_i =v_1\}}=\sum_{i=1}^{n-1}Y_{i+1}=\sum_{j=2}^{n}Y_j,\,n\geq2.
	\end{equation*}
	Also, let $\{Z_{n}\}_{n\in\mathbb{N}_{0}}$ be a sequence of random variables defined as follows:
	\begin{equation*}
		Z_0=
		\begin{cases}
			v_1 R_1\, &\text{if } S_0 =v_1,\\
			-v_2 L_1 \, &\text{if } S_0 =-v_2,
		\end{cases}
	\end{equation*}
	and for $n\in\mathbb{N}$ and $1\leq i\leq n+1$
	\begin{equation*}
		Z_n=
		\begin{cases}
			v_1 R_i \, &\text{if } S_n =v_1\, \text{and}\, Y_1 + M_{n-1}=i-1,\\
			-v_2 L_i \, &\text{if } S_n =-v_2\, \text{and}\, Y_1 + M_{n-1}=n-i+1.
		\end{cases}
	\end{equation*}
	That is, $Z_n$ is the signed distance traveled by the particle during the interval $[T_{n},T_{n+1})$, $n\in\mathbb{N}_{0}$. So, the evolution of the process is described by the following stochastic equation:
	\begin{equation*}
		X(T_{n+1})=X(T_n)+Z_n,\, n\in\mathbb{N}_0.
	\end{equation*}
	
	\subsection{Potential applications}
	A Brownian motion driven by a generalized telegraph process is used to model and forecast ground displacements and their changing tendency during the Campi Flegrei inflation-deflation episodes (see Travaglino {\it et al.} (2018)). Telegraph processes have applications in finance and market modeling. A thorough analysis of the telegraph process and its applications to option pricing is given in Kolesnik and Ratanov (2013).  A new generalization of telegraph process with variable velocities and jumps is studied, and applied to market modeling (see Ratanov (2015)). 
	
	In standard telegraph process, the velocity of a particle changes according to Poissonian arrival. But in real life scenarios, an arrival of Poisson event may not immediately change the direction of particle. A generalized telegraph process with velocity driven by random trials is potentially applicable in such situations. Here, we give two examples as follows:
	
	I. Let $X(t)$ be the stock price of a company at time $t$. The stock price may increase or decrease based on certain events which includes new product arrival, changes in management, mergers, policy decisions, increase in profit, pandemic {\it etc}. Let $\{\mathcal{N}(t)\}_{t\geq0}$ be a counting process which counts such random events. The occurrence of  random event does not reverse the stock price trend immediately. For example, if there are management changes in the company then stock price may increase with probability $p$ or decrease with probability $1-p$. That is, whenever an event occurs in $\{\mathcal{N}(t)\}_{t\geq0}$, a random trial is performed whose output is denoted by $Y_n$ which determines that the stock price will increase or decrease.
	
	
	II. The one dimensional motion of {\it E. Coli} bacteria can be studied by using generalized telegraph process (see Berg (2004), Angelani and Garra (2019)). Let $X(t)$ be the position of {\it E. Coli} bacteria at time $t$. Its interruption by random orientation, known as tumble, is counted by $\{\mathcal{N}(t)\}_{t\geq0}$. Immediately after the tumble occurs {\it E. Coli} bacteria may change its direction and that depends on local chemical concentration and recent chemical history. So, it tends to move in a direction when its recent chemical history suggests that the attractant concentration is increasing in this direction. Thus, this motion of {\it E. Coli} bacteria can be studied by {\it P\'olya urn scheme} of a generalized telegraph process with velocity driven by random trials.
	
	In this paper, we extend the work of Crimaldi {\it et al.} (2013) by considering the Mittag-Leffler and generalized Mittag-Leffler distributed waiting times for {\it Bernoulli scheme} and {\it P\'olya urn scheme}, respectively. As the Mittag-Leffler distribution is heavy tailed, its inclusion in telegraph process makes the resulting process more effective for modeling real world phenomena that generally exhibit long-range dependence.
	
	The paper is organized as follows:

	In Section \ref{SecPrelim}, we give some details on Mittag-Leffler function and its generalizations.
	In Section \ref{SecMLWaitTime}, we study the {\it Bernoulli scheme} of generalized telegraph process by considering $R_n$ and $L_n$ to be Mittag-Leffler distributed. We call this process as the telegraph process with Mittag-Leffler waiting times and velocity driven by iid random trials. The distribution of the sum of independent Mittag-Leffler random variables is derived. For this telegraph process, we obtain the explicit forms of discrete component of probability law and absolutely continuous component of conditional probability law given initial velocity. Moreover, conditional on the initial velocity, expectation of $n$th event time of counting process associated with this telegraph process is obtained. Some plots of absolutely continuous component of its probability law are compared for different sets of parameters.
	In Section \ref{SecGMLWaitTime}, we discuss the {\it P\'olya urn scheme} of generalized telegraph process by considering $R_1$ and $L_1$ to be generalized Mittag-Leffler distributed whereas $R_n$ and $L_n$ to be Mittag-Leffler distributed for all $n\geq2$. The distribution of the sum of independent generalized Mittag-Leffler random variable and Mittag-Leffler random variables is obtained. For this telegraph process, we derive the discrete component of probability law and absolutely continuous component of conditional probability law given initial velocity. Also, conditional on the initial velocity, expectation of $n$th event time of counting process associated with this telegraph process is obtained. Some plots of the absolutely continuous component of its probability law are compared for different parameters.
	
	\section{Preliminaries}\label{SecPrelim}
	First, we set some notations as follows: For any absolutely continuous random variable $X$, let $f_{X}(t)$, $F_{X}(t)$ and $\bar{F}_{X}(t)$ denote its probability density function (pdf), cumulative distribution function (cdf) and tail distribution function, respectively. Also, let $f_Y^{(n)}(t)$, $F_Y^{(n)}(t)$ and $\bar{F}_Y^{(n)}(t)$ be the pdf, cdf and tail distribution function of the convolution of independent random variables $Y_i$, $1 \leq i \leq n$, that is, $Y = Y_1 + Y_2 + \dots + Y_n$. 
	

    Next, we give few definitions and known results.
	
	The three-parameter Mittag-Leffler function which is also known as the Prabhakar function is defined as follows (see Kilbas {\it et al.} (2006)):
	\begin{equation}\label{Mittag12}
		E_{\alpha,\beta}^{\gamma}(t)\coloneqq\sum_{k=0}^{\infty}\frac{(\gamma)_k t^{k}}{k!\Gamma(\alpha k+\beta)},\,t\in\mathbb{R},
	\end{equation}
	where $\alpha>0$, $\beta>0$ and $\gamma>0$. 
	Its Laplace transform is given by
	\begin{equation}\label{LTMittag}
		\mathbb{L}(t^{\beta-1}E_{\alpha,\beta}^{\gamma}(\omega t^{\alpha}))(z)=\frac{z^{\alpha\gamma-\beta}}{(z^{\alpha}-\omega)^{\gamma}},\, \omega\in\mathbb{R},\,z>|\omega|^{1/\alpha},\,t>0.
	\end{equation}
	Here, $\mathbb{R}$ is the set of real numbers and $(\gamma)_k$ denotes the Pochhammer's symbol, that is, $(\gamma)_k={\Gamma(\gamma+k)}/{\Gamma(\gamma)}$. For $\gamma=1$, \eqref{Mittag12} reduces to the two-parameter Mittag-Leffler function denoted by $E_{\alpha,\beta}(\cdot)$. Further, for $\gamma=\beta=1$, we get the one-parameter Mittag-Leffler function. It is denoted by $E_{\alpha}(\cdot)$.

	
	A bi-variate version of the Mittag-Leffler function is defined as follows (for its multivariate version, we refer to Saxena {\it et al.} (2011)):
		\begin{equation}\label{MultGenML}
				E_{(\alpha_1, \alpha_2),\beta}^{(\gamma_1,\gamma_2)}(t_1, t_2)\coloneqq\sum_{k=0}^{\infty}\sum_{l=0}^{\infty}\frac{(\gamma_1)_{k}(\gamma_2)_{l}t_{1}^{k}t_{2}^{l}}{\Gamma(k \alpha_1+l \alpha_2 + \beta)k!l!}.
			\end{equation}
where $\beta \in\mathbb{R}$, $\gamma_i \in\mathbb{R}$ and $\alpha_i >0$, $i=1,2$.
	We call it the bi-variate generalized Mittag-Leffler function.

	Let $X$ be a random variable whose Laplace transform is given by $\mathbb{E}(e^{-zX})=(\lambda/(z^\alpha +\lambda))^{\delta}$, $0<\alpha\leq1$, $\delta>0$ and $\lambda>0$.
    Its pdf is given by
	\begin{equation}\label{GenMLDist}
		f_{X}(t)=\lambda^{\delta}t^{\alpha\delta-1}E_{\alpha,\alpha\delta}^{\delta}(-\lambda t^\alpha),\,t\geq0.
	\end{equation}
	It is known as the generalized Mittag-Leffler random variable (see Jose {\it et al.} (2010)).
	For $\delta=1$, it reduces to the Mittag-Leffler distribution.
	
	
	
	The generalized Wright function also known as the Fox-Wright function is defined as follows (see Kilbas {\it et al.} (2006)):
	\begin{equation}\label{DefFoxWright}
		{}_p\Psi_q\!\left[
		\begin{array}{c}
			(a_1,A_1)\dots(a_p,A_p)\\
			(b_1,B_1)\dots(b_q,B_q)
		\end{array}
		\,\middle|\, t
		\right]
		=
		\sum_{k=0}^{\infty}
		\frac{\prod_{i=1}^{p}\Gamma(a_i+A_i k)}
		{\prod_{j=1}^{q}\Gamma(b_j+B_j k)}
		\frac{t^k}{k!},\,t\in\mathbb{R},
	\end{equation}
	where $a_i$'s, $b_j$'s, $A_i$'s and $B_j$'s are real numbers.
	
	\section{Telegraph process with Mittag-Leffler waiting times and velocity driven by iid random trials}\label{SecMLWaitTime}
	
	Recall the generalized telegraph process from Section \ref{SecStochModel}. Here, for $c=0$, that is, the case of iid random trials, we consider all $R_n$'s and $L_n$'s to be Mittag-Leffler distributed. In other words, we study Case I: {\it Bernoulli scheme} by considering the generalized telegraph process with Mittag-Leffler waiting times. These are heavy tailed distributions that introduces memory effects in the model. 
	
	Let the distribution of $R_{n}$ and $L_{n}$ be given by $\bar{F}_{R_n}(t)=1-\mathrm{Pr}\{R_{n}\leq t\}=E_{\alpha}(-n\lambda t^{\alpha})$ and
	$\bar{F}_{L_n}(t)=1-\mathrm{Pr}\{L_{n}\leq t\}=E_{\alpha}(-n\mu t^{\alpha})$, $t\geq0$, respectively. Here, $\lambda$ and $\mu$ are positive constants, and $0<\alpha\leq1$. So, their pdfs are given by ${f}_{R_n}(t)=n\lambda t^{\alpha-1} E_{\alpha,\alpha}(-n\lambda t^{\alpha})$ and
	${f}_{L_n}(t)=n\mu t^{\alpha-1}E_{\alpha,\alpha}(-n\mu t^{\alpha})$, $t\geq0$, respectively (see Kataria and Vellaisamy (2019)).
	
	We call this process as the telegraph process with Mittag-Leffler waiting times and velocity driven by  iid random trials, and denote it by $\{(X_{\alpha}(t),V_{\alpha}(t))\}_{t\geq0}$. 
	
	The following result will be used. Its proof follows similar lines to that of Theorem 2.2 of Kataria and Vellaisamy (2019).
	
	\begin{proposition}\label{Xpdf}
		Let $\{X_i\}_{1\leq i\leq n}$ be a sequence of independent random variables with distribution $\bar{F}_{X_{i}}(t)=E_{\alpha}(-\lambda_{i}t^\alpha)$. Then, the pdf and cdf of $X=X_{1}+X_{2}+\dots+X_{n}$ are given by
		\begin{equation*}
			f_{X}(t)=t^{\alpha-1}\sum_{i=1}^{n}{\lambda_{i}E_{\alpha,\alpha}(-\lambda_{i}t^\alpha)}\underset{ k\neq i}{\prod_{k=1}^{n}}\frac{\lambda_{k}}{\lambda_{k}-\lambda_{i}}
		\end{equation*}
		and 
		\begin{equation*}
			F_{X}(t)=1-\sum_{i=1}^{n}E_{\alpha}(-\lambda_{i}t^\alpha)\prod_{\substack{k=1 \\ k \neq i}}^{n}\frac{\lambda_{k}}{\lambda_{k}-\lambda_i},
		\end{equation*}
		respectively.
	\end{proposition}
	\begin{proof}
		As $X_{i}$'s are independent, the Laplace transform of pdf of $X$ is given by
		\begin{align}\label{LTXpdf}
			\mathbb{E}(e^{-zX})
			&=\prod_{i=1}^{n}\frac{\lambda_{i}}{z^\alpha+\lambda_{i}}\nonumber\\
			&=\Big(\prod_{j=1}^{n}\lambda_{j}\Big)\sum_{i=1}^{n}\frac{1}{z^\alpha +\lambda_{i}} \underset{ k\neq i}{\prod_{k=1}^{n}}\frac{1}{\lambda_{k}-\lambda_{i}}.
		\end{align}
		On taking the inverse Laplace transform of \eqref{LTXpdf} and by using \eqref{LTMittag}, we get the pdf of $X$ as
		\begin{equation*}\label{XpdfHalf}
			f_{X}(t)=t^{\alpha-1}\sum_{i=1}^{n}{\lambda_{i}E_{\alpha,\alpha}(-\lambda_{i}t^\alpha)}\prod_{\substack{k=1 \\ k \neq i}}^{n}\frac{\lambda_{k}}{\lambda_{k}-\lambda_i}
		\end{equation*}
	   and its cdf as
		\begin{align*}
			F_{X}(t)
			&=\int_{0}^{t}f_{X}(s)ds\nonumber\\
			&=\sum_{i=1}^{n}\lambda_{i}\Big(\prod_{\substack{k=1 \\ k \neq i}}^{n}\frac{\lambda_{k}}{\lambda_{k}-\lambda_i}\Big)\int_{0}^{t}s^{\alpha-1}E_{\alpha,\alpha}(-\lambda_{i}s^\alpha)\mathrm{d}s\nonumber\\
			&=t^{\alpha}\sum_{i=1}^{n}{\lambda_{i}E_{\alpha,\alpha+1}(-{\lambda_{i}}t^{\alpha})}\prod_{\substack{k=1 \\ k \neq i}}^{n}\frac{\lambda_{k}}{\lambda_{k}-\lambda_i} \ \ (\text{see Gorenflo {\it et al.} (2014), Eq. (4.4.4))}\nonumber\\
			&=\sum_{i=1}^{n}({1-E_{\alpha}(-\lambda_{i}t^\alpha)})\prod_{\substack{k=1 \\ k \neq i}}^{n}\frac{\lambda_{k}}{\lambda_{k}-\lambda_i}\nonumber\\
			&=1-\sum_{i=1}^{n}E_{\alpha}(-\lambda_{i}t^\alpha)\prod_{\substack{k=1 \\ k \neq i}}^{n}\Big(\frac{\lambda_{k}}{\lambda_{k}-\lambda_i}\Big),
		\end{align*}
		where we have used $\sum_{i=1}^{n}\prod_{\substack{k=1 \\ k \neq i}}^{n}\frac{\lambda_{k}}{\lambda_{k}-\lambda_i}=1$ in the last step. This completes the proof.
	\end{proof}
	
	From Proposition \ref{Xpdf}, the pdf and cdf of $R^{(n)}=R_{1}+R_{2}+\dots+R_{n}$ are given by
	\begin{equation}\label{pdfRnhalf}
		f^{(n)}_{R}(t)=n!\lambda t^{\alpha-1}\sum_{i=1}^{n}E_{\alpha,\alpha}(-\lambda i t^{\alpha})\underset{ k\neq i}{\prod_{k=1}^{n}}\frac{1}{k-i}
	\end{equation}
	and
	\begin{equation}\label{cdfRnhalf}
		F^{(n)}_{R}(t)=1-\sum_{i=1}^{n}E_{\alpha}(-\lambda it^{\alpha})\prod_{\substack{k=1 \\ k \neq i}}^{n}\frac{k}{k-i},
	\end{equation}
	respectively.
	Note that
	\begin{equation}\label{ProdRelation}
		\underset{ k\neq i}{\prod_{k=1}^{n}}(k-i)=(-1)^{i-1}(i-1)!(n-i)!.
	\end{equation}
	By using \eqref{ProdRelation} in \eqref{pdfRnhalf} and \eqref{cdfRnhalf}, the pdf and cdf of $R^{(n)}$ reduces to
	\begin{equation}\label{pdfRn}
		f^{(n)}_{R}(t)=n\lambda t^{\alpha-1} \sum_{i=0}^{n-1}(-1)^i\binom{n-1}{i}E_{\alpha,\alpha}(-\lambda(i+1)t^{\alpha})
	\end{equation}
	and
	\begin{equation}\label{cdfRn}
		F^{(n)}_{R}(t)=\sum_{i=0}^{n}(-1)^i\binom{n}{i}E_{\alpha}(-i\lambda t^{\alpha}),
	\end{equation}
	respectively.
	
	Similarly, the pdf and cdf of $L^{(n)}=L_{1}+L_{2}+\dots+L_{n}$ are given by
	\begin{equation}\label{pdfLn}
		f^{(n)}_{L}(t)=n\mu t^{\alpha-1} \sum_{i=0}^{n-1}(-1)^i\binom{n-1}{i}E_{\alpha,\alpha}(-\mu(i+1)t^{\alpha})
	\end{equation}
	and 
	\begin{equation}\label{cdfLn}
		F^{(n)}_{L}(t)=\sum_{i=0}^{n}(-1)^i\binom{n}{i}E_{\alpha}(-i\mu t^{\alpha}),
	\end{equation}
	respectively.
	\begin{remark}
		For $\alpha=1$, the pdfs and cdfs of $R^{(n)}$ and $L^{(n)}$ reduces to 	$f_{R}^{(n)}=n\lambda e^{-\lambda t}(1-e^{-\lambda t})^{n-1}$, $f_{L}^{(n)}=n\mu e^{-\mu t}(1-e^{-\mu t})^{n-1}$,
		$F_{R}^{(n)}=(1-e^{-\lambda t})^{n}$,
		$F_{L}^{(n)}=(1-e^{-\mu t})^{n}$ 
		which agrees with Eq. (24) and Eq. (25) of Crimaldi {\it et al.} (2013). That is, for $\alpha=1$, the telegraph process with Mittag-Leffler waiting times and velocity driven by iid random trials reduces to the generalized telegraph process with velocity driven by random trials under {\it Bernoulli scheme}.
	\end{remark}

	Note that the particle is at position $v_{1}t$ or $-v_{2}t$ on the real line with positive probability if the particle does not change its direction up to time $t$. Conditional on the occurrence of at least one direction change by time $t$, the random variable $X_{\alpha}(t)$ admits a continuous distribution on the interval $(-v_2t,v_1t)$. So, $X_{\alpha}(t)$, $t\geq0$ has mixed distribution. 
	
	In the next result, we derive the discrete component of the probability law of $X_{\alpha}(t)$, $t\geq0$.
	
	\begin{theorem}\label{ThmDiscreteProb}
		The probabilities that the particle is at positions $v_{1}t$ and $-v_2t$ are
		\begin{equation*}\label{DiscProbXtv1}
			\mathrm{Pr}\{X_{\alpha}(t)=v_1 t\}=\sum_{n=0}^{\infty}p^{n+1}\sum_{k=0}^{n}(-1)^k \binom{n}{k}E_{\alpha}(-\lambda(k+1)t^{\alpha})
		\end{equation*}
		and
		\begin{equation*}\label{DiscProbXtv2}
			\mathrm{Pr}\{X_{\alpha}(t)=-v_2 t\}=\sum_{n=0}^{\infty}(1-p)^{n+1}\sum_{k=0}^{n}(-1)^k \binom{n}{k}E_{\alpha}(-\mu(k+1)t^{\alpha}),
		\end{equation*}
		respectively.
	\end{theorem}
	\begin{proof}
	From \eqref{cdfRn}, we have
		\begin{align}\label{Rncdfdiff}
			F^{(n)}_{R}(t)-F^{(n+1)}_{R}(t)
			&=\sum_{k=0}^{n}(-1)^k\binom{n}{k}E_{\alpha}(-\lambda k t^{\alpha})-\sum_{k=0}^{n+1}(-1)^k\binom{n+1}{k}E_{\alpha}(-\lambda kt^{\alpha})\nonumber\\
			&=\sum_{k=1}^{n}(-1)^{k-1} \Big(\binom{n+1}{k}-\binom{n}{k}\Big)E_{\alpha}(-\lambda kt^{\alpha})+(-1)^n E_{\alpha}(-\lambda(n+1)t^{\alpha})\nonumber\\
			&=\sum_{k=1}^{n}(-1)^{k-1}\binom{n}{k-1}E_{\alpha}(-\lambda kt^{\alpha})+(-1)^n E_{\alpha}(-\lambda(n+1)t^{\alpha})\nonumber\\
			&=\sum_{k=0}^{n}(-1)^k \binom{n}{k}E_{\alpha}(-\lambda(k+1)t^{\alpha}).
		\end{align}
	Similarly, by using \eqref{cdfLn}, we get
	\begin{equation}\label{Lncdfdiff}
			F^{(n)}_{L}(t)-F^{(n+1)}_{L}(t)= \sum_{k=0}^{n}(-1)^k \binom{n}{k}E_{\alpha}(-\mu(k+1)t^{\alpha}).
	\end{equation}
	By using \eqref{Rncdfdiff} and \eqref{Lncdfdiff} in Proposition 1 of Crimaldi {\it et al.} (2013), we get the required result.
	\end{proof}
	
	\begin{remark}
		On substituting $\alpha=1$ in Theorem \ref{ThmDiscreteProb}, we get
		\begin{align*}
			\mathrm{Pr}\{X_{1}(t)=v_1 t\}&=\sum_{n=0}^{\infty}p^{n+1}\sum_{k=0}^{n}(-1)^k \binom{n}{k}e^{-\lambda(k+1)t}\\
			&=pe^{-\lambda t}\sum_{n=0}^{\infty}p^{n}(1-e^{-\lambda t})^{n}
			=\frac{pe^{-\lambda t}}{1-p(1-e^{-\lambda t})}
		\end{align*}
		and 
		\begin{equation*}
			\mathrm{Pr}\{X_{1}(t)=-v_{2}t\}=\frac{(1-p)e^{-\mu t}}{1-(1-p)(1-e^{-\mu t})}
		\end{equation*}
		which agree with the corresponding result for generalized telegraph process with velocity driven by random trials: {\it Bernoulli scheme} (see Crimaldi {\it et al.} (2013), Proposition 3).
	\end{remark}
	
	Let $H^{\alpha}_{\beta,\eta}(a,b;t)= t^{\beta-1}E_{\alpha,\beta}(at^\alpha)*t^{\eta-1}E_{\alpha,\eta}(bt^\alpha)$, $t\geq0$ be the convolution of two-parameter Mittag-Leffler functions. So,
	\begin{equation}\label{Hfuncdef}
		H^{\alpha}_{\beta,\eta}(a,b;t)=\int_{0}^{t}s^{\beta-1}E_{\alpha,\beta}(as^\alpha)(t-s)^{\eta-1}E_{\alpha,\eta}(b(t-s)^\alpha)\mathrm{d}s,\,t\geq0.
	\end{equation}
	By using \eqref{LTMittag}, the Laplace transform of \eqref{Hfuncdef} is given by
	\begin{align}\label{HfuncLT}
		\mathbb{L}(H^{\alpha}_{\beta,\eta}(a,b;t))(z)
		&=\frac{z^{2\alpha-\beta-\eta}}{(z^\alpha-a)(z^\alpha-b)}\nonumber\\
		&=\frac{1}{a-b}\Big(\frac{az^{\alpha-\beta-\eta}}{z^\alpha-a}-\frac{bz^{\alpha-\beta-\eta}}{z^\alpha-b}\Big),\,a\neq b.
	\end{align}
	On taking the inverse Laplace transform of \eqref{HfuncLT}, and from Eq. (2.21) of Kilbas {\it et al.} (2004), we have
	\begin{align}\label{HfuncForm}
		H^{\alpha}_{\beta,\eta}(a,b;t)=
		\begin{cases}
			t^{\beta+\eta-1}E_{\alpha,\beta+\eta}^{2}(at^\alpha),\, a=b,\\[0.3em]
			\frac{t^{\beta+\eta-1}}{a-b}(aE_{\alpha,\beta+\eta}(at^\alpha)-bE_{\alpha,\beta+\eta}(bt^\alpha)), \, a\neq b.
		\end{cases}
	\end{align}
	
	The conditional probability density of $X_{\alpha}(t)$ given $V_{\alpha}(0)=v\in\{v_{1},-v_2\}$ can be described in terms of the conditional forward and backward transition pdfs as follows: $p(x,t|v)=f(x,t|v)+b(x,t|v)$ (see Crimaldi {\it et al.} (2013), p. 1116). 
	
	For $x\in\mathbb{R}$, $t\geq0$ and $v\in\{v_{1},-v_{2}\}$, the conditional forward and backward pdfs are given by
	\begin{equation*}
		f(x,t|v)=\frac{\partial}{\partial x}\mathrm{Pr}\{X_{\alpha}(t)\leq x, V_{\alpha}(t)=v_{1}| V_{\alpha}(0)=v\}
	\end{equation*}
	and 
	\begin{equation*}
		b(x,t|v)=\frac{\partial}{\partial x}\mathrm{Pr}\{X_{\alpha}(t)\leq x, V_{\alpha}(t)=-v_{2}| V_{\alpha}(0)=v\},
	\end{equation*}
	respectively.
	Also, we have
	\begin{equation}\label{fxtsum}
		f(x,t|v)=\sum_{n=1}^{\infty}f_{n}(x,t|v)
	\end{equation}
	and
	\begin{equation}\label{bxtsum}
		b(x,t|v)=\sum_{n=1}^{\infty}b_{n}(x,t|v),
	\end{equation}
	where
	\begin{equation*}
		f_{n}(x,t|v)=\frac{\partial}{\partial x}\mathrm{Pr}\{X_{\alpha}(t)\leq x, V_{\alpha}(t)=v_{1}, \mathcal{N}(t)=n| V_{\alpha}(0)=v\}
	\end{equation*}
	and
	\begin{equation*}
		b_{n}(x,t|v)=\frac{\partial}{\partial x}\mathrm{Pr}\{X_{\alpha}(t)\leq x, V_{\alpha}(t)=-v_{2}, \mathcal{N}(t)=n| V_{\alpha}(0)=v\}.
	\end{equation*}
	
	In the next result, we obtain the forward and backward densities of conditional probability law of $X_{\alpha}(t)$, $t\geq0$.
	
	\begin{theorem}\label{ThmAbsolutelyCtsProb}
		The absolutely continuous components of the conditional probability law of $X_{\alpha}(t)$, $t\geq0$ are given by its forward and backward densities as follows:
		\begin{align}\label{fxtv}
			f(x,t|v_1)
			&=\frac{\mu(t-\theta)^{\alpha-1}}{v_1 +v_2}\sum_{n=2}^{\infty}\sum_{k=0}^{n-2}\sum_{i=0}^{n-k-2}\sum_{j=0}^{k+1}(-1)^{i+j}\binom{n-1}{k}\binom{n-k-2}{i}\binom{k+1}{j}(n-k-1)\nonumber\\
			&\ \ \cdot p^{k+1}(1-p)^{n-k-1}E_{\alpha}(-\lambda(j+1)\theta^{\alpha})E_{\alpha,\alpha}(-\mu(i+1)(t-\theta)^{\alpha}),
		\end{align}
		and 
		\begin{align}\label{bxtv}
			b(x,t|v_1)
			&=\frac{\lambda\theta^{\alpha-1}}{v_1 +v_2}\sum_{n=1}^{\infty}\sum_{k=0}^{n-1}\sum_{i=0}^{k}\sum_{j=0}^{n-k-1}(-1)^{i+j}\binom{n-1}{k}\binom{k}{i}\binom{n-k-1}{j}(k+1)p^{k}(1-p)^{n-k}\nonumber\\
			&\ \ \cdot E_{\alpha}(-\mu(j+1)(t-\theta)^{\alpha})E_{\alpha,\alpha}(-\lambda(i+1)\theta^{\alpha}),\, -v_2 t< x <v_1t,
		\end{align}
		where $\theta={(v_2 t+x)}/{(v_1 + v_2)}$.
	\end{theorem}
	\begin{proof}
		Consider the following integral:
		{\small\begin{align}\label{IntIfn}
			\int_{t-\theta}^{t}&f^{(k+1)}_{R}(s-t+\theta)
			\bar{F}_{R_{k+2}}(t-s)\mathrm{d}s\nonumber\\
			&=\int_{0}^{\theta}f^{(k+1)}_{R}(y)
			\bar{F}_{R_{k+2}}(\theta-y)\mathrm{d}y\nonumber\\
			&=\lambda(k+1)\sum_{j=0}^{k}(-1)^{j}\binom{k}{j}\int_{0}^{\theta}y^{\alpha-1}E_{\alpha,\alpha}(-\lambda (j+1)y^{\alpha})E_{\alpha}(-\lambda(k+2)(\theta-y)^{\alpha})\mathrm{d}y\nonumber\\
			&=\lambda(k+1)\sum_{j=0}^{k}(-1)^{j}\binom{k}{j}H^{\alpha}_{\alpha,1}(-\lambda(j+1),-\lambda(k+2);\theta)\ \ (\text{by using \eqref{Hfuncdef}})\nonumber\\
			&=\lambda(k+1)\sum_{j=0}^{k}\frac{(-1)^{j}}{k+1-j}\binom{k}{j}{\theta^\alpha\big((k+2)E_{\alpha,\alpha+1}(-\lambda(k+2)\theta^\alpha)-(j+1)E_{\alpha,\alpha+1}(-\lambda(j+1)\theta^\alpha)\big)}\,\nonumber\\
			&\hspace{12cm}
			(\text{by using \eqref{HfuncForm}})\nonumber\\
			&=(k+1)\sum_{j=0}^{k}\frac{(-1)^{j}}{k+1-j}\binom{k}{j}{E_{\alpha}(-\lambda(j+1)\theta^\alpha)-E_{\alpha}(-\lambda(k+2)\theta^\alpha)\big)}\nonumber\\
			&=\sum_{j=0}^{k+1}(-1)^j\binom{k+1}{j}E_{\alpha}(-\lambda(j+1)\theta^\alpha).
		\end{align}}
	By using \eqref{IntIfn} in Eq. (15) of Crimaldi {\it et al.} (2013), we obtain
	\begin{align}\label{fnxtc}
		f_{n}(x,t|v_1)&=\frac{\mu(t-\theta)^{\alpha-1}}{v_1+v_2}\sum_{k=0}^{n-2}\binom{n-1}{k}p^{k+1}(1-p)^{n-k-1}(n-k-1)\sum_{i=0}^{n-k-2}(-1)^i\binom{n-k-2}{i}\nonumber\\
		&\ \ \cdot E_{\alpha,\alpha}(-\mu (i+1)(t-\theta)^\alpha)\sum_{j=0}^{k+1}(-1)^j\binom{k+1}{j}E_{\alpha}(-\lambda(j+1)\theta^\alpha),
	\end{align}
	where we have used following result (see Crimaldi {\it et al.} (2013), Eq. (45)):
	\begin{equation*}
		\mathrm{Pr}\{M_{n-1}=k, S_{n} =v_1|S_0 =v_1\}=\binom{n-1}{k}p^{k+1}(1-p)^{n-k-1}.
	\end{equation*}
	By using \eqref{fnxtc} in \eqref{fxtsum}, we obtain the required result \eqref{fxtv}. 
	
	Similarly,
	\begin{equation}\label{IntIbn}
		\int_{\theta}^{t}f^{(n-k-1)}_{L}(s-\theta)
		\bar{F}_{L_{n-k}}(t-s)\mathrm{d}s=\sum_{j=0}^{n-k-1}(-1)^j\binom{n-k-1}{j}E_{\alpha}(-\mu(j+1)(t-\theta)^\alpha).
	\end{equation}
	By using \eqref{IntIbn} and Eq. (16) of Crimaldi {\it et al.} (2013), we obtain
	\begin{align}\label{b1xtc}
		b_{1}(x,t|v_1)=\frac{\lambda\theta^{\alpha-1}(1-p)}{v_1 +v_2}E_{\alpha,\alpha}(-\lambda\theta^\alpha)E_{\alpha}(-\mu(t-\theta)^\alpha),
	\end{align}
	and for $n\geq2$, we get
	\begin{align}\label{bnxtc}
		b_{n}(x,t|v_1)&=\frac{\lambda\theta^{\alpha-1}}{v_1 +v_2}\sum_{k=0}^{n-1}\binom{n-1}{k}p^{k}(1-p)^{n-k}(k+1)\sum_{i=0}^{k}(-1)^{i}\binom{k}{i}E_{\alpha,\alpha}(-\lambda(i+1)\theta^\alpha)\nonumber\\
		&\ \ \cdot\sum_{j=0}^{n-k-1}(-1)^j\binom{n-k-1}{j}E_{\alpha}(-\mu(j+1)(t-\theta)^\alpha),
	\end{align}
	where we have used (see Crimaldi {\it et al.} (2013), Eq. (46))
		\begin{equation*}
		\mathrm{Pr}\{M_{n-1}=k, S_{n}=-v_2|S_{0}=v_1\}=\binom{n-1}{k}p^{k}(1-p)^{n-k}.
	\end{equation*}
	By using \eqref{b1xtc} and \eqref{bnxtc} in \eqref{bxtsum}, we obtain the required result \eqref{bxtv}.
	\end{proof}
	
	\begin{remark}
		On substituting $\alpha=1$ in \eqref{fxtv}, we get
		\begin{align*}
			f(x,t|v_1)
			&=\frac{\mu}{v_1 +v_2}\sum_{n=2}^{\infty}\sum_{k=0}^{n-2}\binom{n-1}{k}p^{k+1}(1-p)^{n-k-1}(n-k-1)\\
			&\hspace{2.8cm} \cdot \sum_{i=0}^{n-k-2}(-1)^{i}\binom{n-k-2}{i} e^{-\mu(i+1)(t-\theta)}\sum_{j=0}^{k+1}(-1)^j \binom{k+1}{j}e^{-\lambda(j+1)\theta}\\
			&=\frac{\mu p(1-p)}{v_1 +v_2}(e^{\lambda\theta}-1) e^{-\mu (t-\theta)-2\lambda\theta}\sum_{n=2}^{\infty}\sum_{k=0}^{n-2}\binom{n-1}{k}p^{k+1}(1-p)^{n-k-1}\\
			&\hspace{5.7cm} \cdot (n-k-1)(1-e^{-\mu i(t-\theta)})^{n-k-2}(1-e^{-\lambda\theta})^{k+1}\\
			&=\frac{\mu p(1-p)(e^{\lambda\theta}-1)e^{\mu(t+\theta)}}{(v_1+v_2)\big((1-p)e^{(\lambda+\mu)\theta}+pe^{\mu t}\big)^2}
		\end{align*}
		which agrees with the corresponding result for generalized telegraph process with velocity driven by random trials: {\it Bernoulli scheme} (see Crimaldi {\it et al.} (2013), Theorem 2). A similar reduction holds true on substituting $\alpha=1$ in \eqref{bxtv}.
	\end{remark}
	
	\begin{remark}
		The absolutely continuous components of the conditional probability law of $X_{\alpha}(t)$, $t\geq0$ are given by its forward and backward densities as follows:
		\begin{align*}
			f(x,t|-v_2)
			&=\frac{\mu(t-\theta)^{\alpha-1}}{v_1 +v_2}\sum_{n=1}^{\infty}\sum_{k=0}^{n-1}\sum_{i=0}^{k}\sum_{j=0}^{n-k-1}(-1)^{i+j}\binom{n-1}{k}\binom{k}{i}\binom{n-k-1}{j}(k+1)\nonumber\\
			&\ \ \cdot (1-p)^{k} p^{n-k} E_{\alpha,\alpha}(-\mu(i+1)(t-\theta)^{\alpha})E_{\alpha}(-\lambda(j+1)\theta^{\alpha})
		\end{align*}
		and
		\begin{align*}
			b(x,t|-v_2)
			&=\frac{\lambda\theta^{\alpha-1}}{v_1 +v_2}\sum_{n=2}^{\infty}\sum_{k=0}^{n-2}\sum_{i=0}^{n-k-2}\sum_{j=0}^{k+1}(-1)^{i+j}\binom{n-1}{k}\binom{n-k-2}{i}\binom{k+1}{j}(n-k-1)\nonumber\\
			&\ \ \cdot (1-p)^{k+1}p^{n-k-1}E_{\alpha,\alpha}(-\lambda(i+1)\theta^{\alpha})E_{\alpha}(-\mu(j+1)(t-\theta)^{\alpha}),\, -v_2 t< x <v_1t.
		\end{align*}
		If 	$\mathrm{Pr}\{V_{\alpha}(0)=v_1\}=\mathrm{Pr}\{V_{\alpha}(0)=-v_2\}=0.5$ then
		\begin{equation}\label{pxt1/2}
			p(x,t)=\frac{1}{2}(p(x,t|v_1)+p(x,t|-v_2)).
		\end{equation}
	\end{remark}

    \begin{figure}
    	\centering
    	\includegraphics[width=1\textwidth]{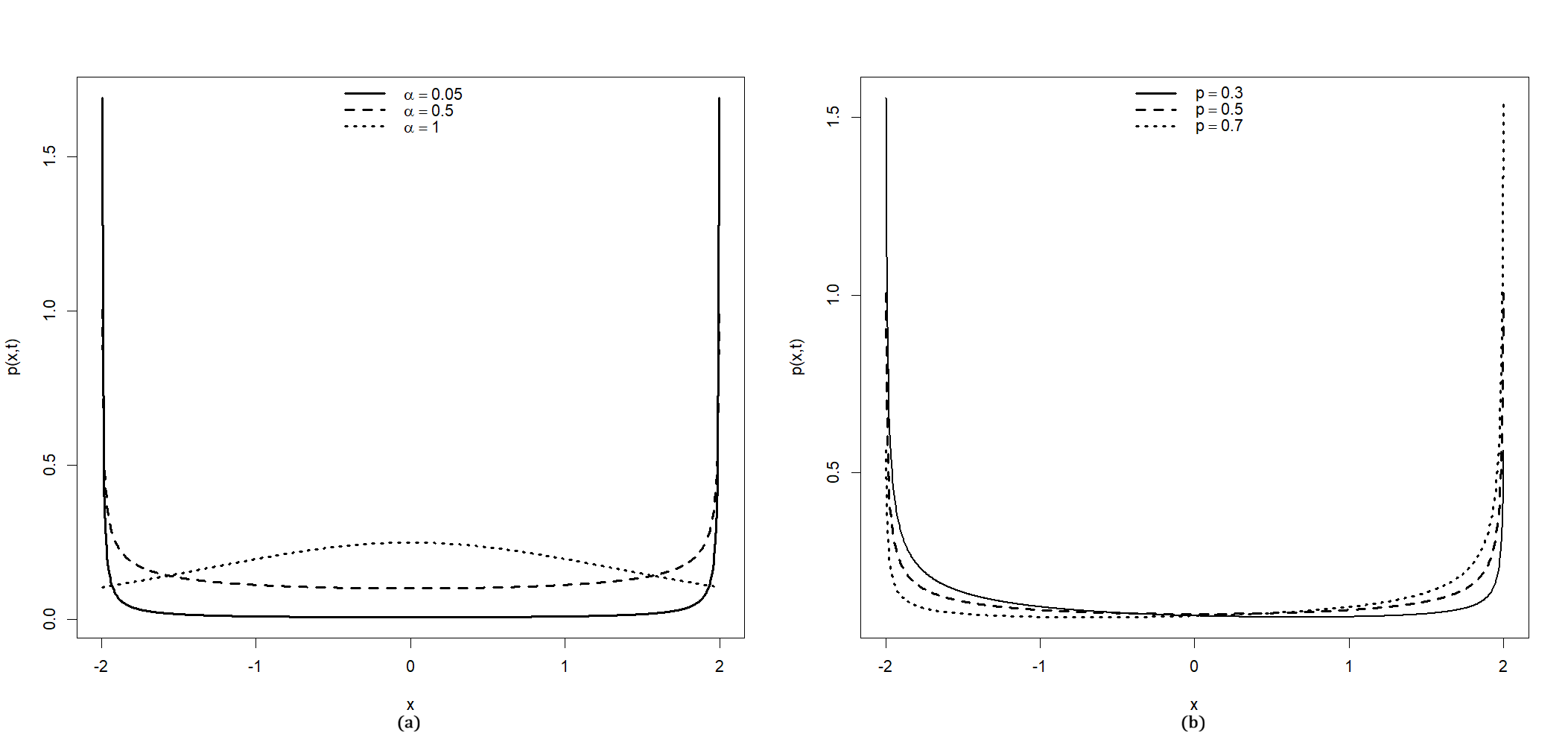}
    	\caption{Plots of $p(x,t)$ (a) for different values of $\alpha$ with $t=2$, $v_1=v_2=1$, $\lambda=\mu=1$ and $p=0.5$, (b) for different values of $p$ with $t=2$, $v_1=v_2=1$, $\lambda=\mu=1$ and $\alpha=0.5$.}
    	\label{fig2} 
    \end{figure}
    \begin{figure}
    	\centering
    	\includegraphics[width=1\textwidth]{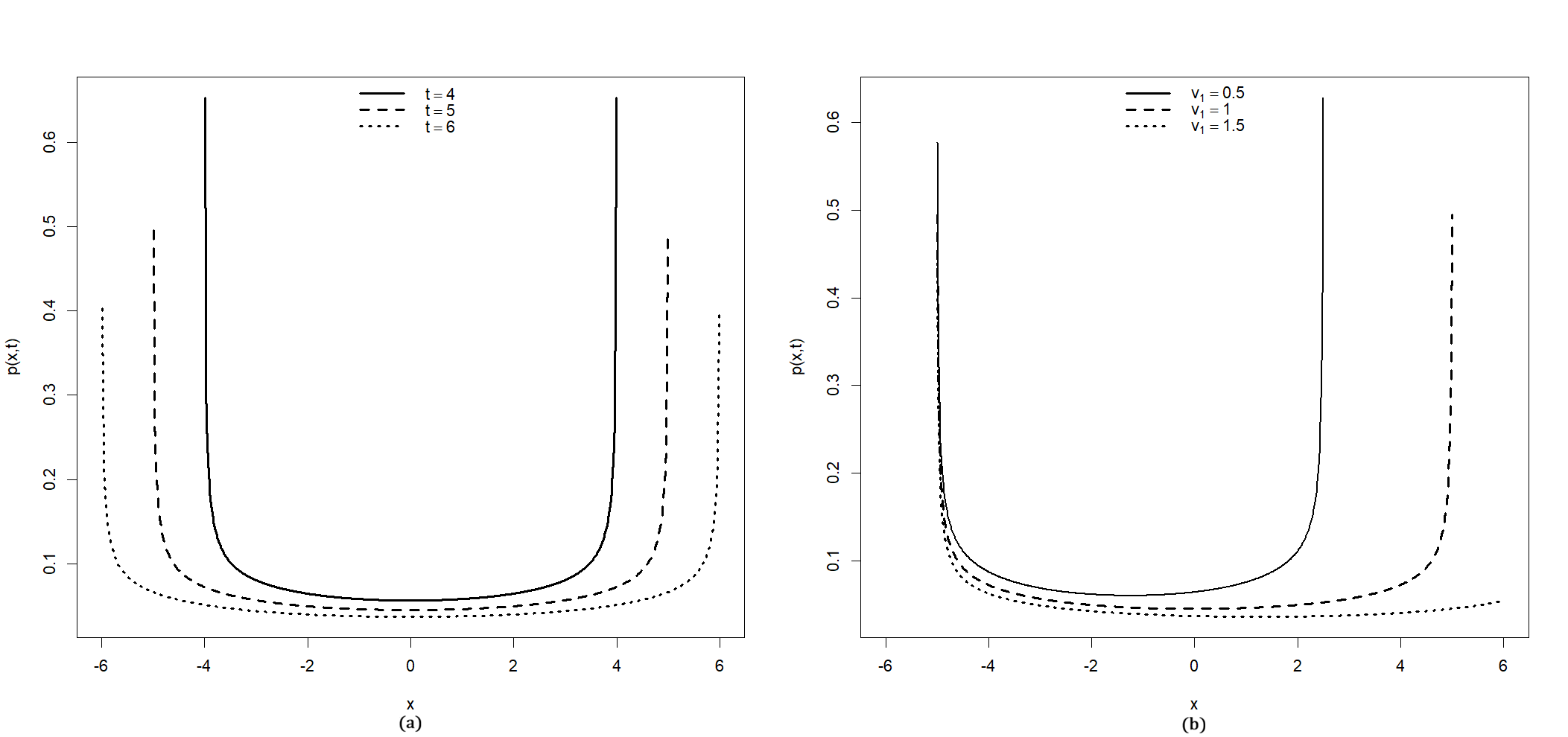}
    	\caption{Plots of $p(x,t)$ (a) for different values of $t$ with $v_1=v_2=1$, $\lambda=\mu=1$, $p=0.5$ and $\alpha=0.5$, (b) for different values of $v_1$ with $t=5$, $v_2=1$, $\lambda=\mu=1$, $\alpha=0.5$ and $p=0.5$.}
    	\label{fig3} 
    \end{figure}
    
    In Figure \ref{fig2}(a), the plots of $p(x,t)$ given in \eqref{pxt1/2} are compared for different values of $\alpha$. It is observed that when $\alpha$ decreases the plot gets peaked around boundaries of support of $X_{\alpha}(t)$. For $\alpha=1$, it shows similar behavior to that observed in Figure 3 of Crimaldi {\it et al.} (2013). In Figure \ref{fig2}(b), the plots of $p(x,t)$ given in \eqref{pxt1/2} are compared for different values of $p$. It is observed that as the probability $p$ of choosing velocity $v_1$ increases, the plot becomes relatively more concentrated near right side of the origin as expected. In Figure \ref{fig3}(a), the plots of $p(x,t)$ given in \eqref{pxt1/2} are compared for different values of $t$. It is observed that support of $X_{\alpha}(t)$ increases with increase in $t$. In Figure \ref{fig3}(b), the plots of $p(x,t)$ given in \eqref{pxt1/2} are compared for different values of $v_1$. As the velocity $v_1$ increases, the support of ${X}_{\alpha}(t)$ increases towards right.

	The conditional distribution of the $n$th event time of counting process $\{\mathcal{N}(t)\}_{t\geq0}$ given that the initial velocity is $v_1$ is given by (see Crimaldi {\it et al.} (2013), Eq. (18))
	\begin{align}\label{ConditionalDistTn}
		F_{T_n|S_{0}}(t|v_1)&\coloneqq\mathrm{Pr}\{T_{n}\leq t|S_{0}=v_1\}\nonumber\\
		&=\sum_{k=0}^{n-1}\mathrm{Pr}\{M_{n-1}=k|S_{0}=v_1\}\mathrm{Pr}\{R^{(k+1)}+L^{(n-k-1)}\leq t\},
	\end{align}
	where 
	{\small\begin{equation*}\label{ProbRL}
		\mathrm{Pr}\{R^{(k+1)}+L^{(n-k-1)}\leq t\}=\int_{0}^{t}F_{R}^{(k+1)}(t-s)f_{L}^{(n-k-1)}(s)\mathrm{d}s=\int_{0}^{t}F_{L}^{(n-k-1)}(t-s)f_{R}^{(k+1)}(s)\mathrm{d}s.
	\end{equation*}}

	\begin{lemma}\label{Tndist}
		Conditional on the event that the initial velocity is $v_1$, the distribution of $n$th event time of counting process $\{\mathcal{N}(t)\}_{t\geq0}$ associated with $X_{\alpha}(t)$ is given by
		\begin{align*}
			F_{T_n|S_{0}}(t|v_1)&=\lambda\sum_{k=0}^{n-1}\binom{n-1}{k}(k+1)p^{k}(1-p)^{n-k-1}\nonumber\\
			&\hspace{2.5cm} \cdot \sum_{i=0}^{n-k-1}\sum_{j=0}^{k}(-1)^{i+j}\binom{n-k-1}{i}\binom{k}{j} H^{\alpha}_{\alpha,1}(-\lambda(j+1),-\mu i;t).
		\end{align*}
	\end{lemma}
	\begin{proof}
		Consider the following integral:

        {\small\begin{align}\label{IntTnS0}
        	\int_{0}^{t}
        	&F_{L}^{(n-k-1)}(t-s)f_{R}^{(k+1)}(s)\mathrm{d}s\nonumber\\
        	&=\lambda(k+1)\sum_{i=0}^{n-k-1}\sum_{j=0}^{k}(-1)^{i+j}\binom{n-k-1}{i}\binom{k}{j}\int_{0}^{t}E_{\alpha}(-\mu i(t-s)^\alpha)s^{\alpha-1}E_{\alpha,\alpha}(-\lambda(j+1)s^\alpha)\mathrm{d}s\nonumber\\
        	&=\lambda(k+1)\sum_{i=0}^{n-k-1}\sum_{j=0}^{k}(-1)^{i+j}\binom{n-k-1}{i}\binom{k}{j}H^{\alpha}_{\alpha,1}(-\lambda(j+1),-\mu i;t),
        \end{align}}
        where we have used \eqref{Hfuncdef}.
		Also, we have (see Crimaldi {\it et al.} (2013), p. 1132)
		\begin{equation}\label{ProbC0}
			\mathrm{Pr}\{M_{n-1}=k|S_{0}=v_1\}=\binom{n-1}{k}p^{k}(1-p)^{n-k-1}.
		\end{equation}
	   By using \eqref{IntTnS0} and \eqref{ProbC0} in \eqref{ConditionalDistTn}, we get the required result.
	\end{proof}
	
	\begin{remark}
		By using \eqref{HfuncForm} in Lemma \ref{Tndist}, we have
		{\small\begin{align*}
			F_{T_n|S_{0}}(t|v_1)&=\lambda\sum_{k=0}^{n-1}\binom{n-1}{k}(k+1)p^{k}(1-p)^{n-k-1}\sum_{i=0}^{n-k-1}\sum_{j=0}^{k}(-1)^{i+j}\binom{n-k-1}{i}\binom{k}{j}\nonumber\\
			&\ \ \cdot\Big(\mathds{1}_{\{\mu i=\lambda(j+1)\}}t^\alpha E_{\alpha,\alpha+1}^{2}(-\mu it^\alpha)+\mathds{1}_{\{\mu i\neq\lambda(j+1)\}}\Big(\frac{E_{\alpha}(-\lambda(j+1)t^{\alpha})-E_{\alpha}(-\mu i t^{\alpha})}{\mu i-\lambda(j+1)}\Big)\Big).
		\end{align*}}
	\end{remark}
	
	Let $G^{\alpha}_{\beta,\eta,\sigma}(a,b,c;t) = H^{\alpha}_{\beta,\eta}(a,b;)*t^{\sigma-1}E_{\alpha,\sigma}(ct^\alpha)$, $t\geq0$. That is,
	\begin{equation}\label{GfuncDef}
		 G^{\alpha}_{\beta,\eta,\sigma}(a,b,c;t)=\int_{0}^{t}H^{\alpha}_{\beta,\eta}(a,b;t-s)s^{\sigma-1}E_{\alpha,\sigma}(cs^\alpha)\mathrm{d}s.
	\end{equation}
	Let us derive the explicit expression of $G^{\alpha}_{\beta,\eta,\sigma}(a,b,c;t)$ which can be obtained in five distinct cases based on the equality of $a$, $b$ and $c$. For the case $a=b=c$, we have 
	\begin{equation*}
		G^{\alpha}_{\beta,\eta,\sigma}(a,a,a;t)=t^{\beta+\eta+\sigma-1}E_{\alpha,\beta+\eta+\sigma}^{3}(at^\alpha),
	\end{equation*}
	which is obtained by using \eqref{HfuncForm} and Eq. (2.21) of Kilbas {\it et al.} (2004). Now, for the case $a=b\neq c$, we have
	\begin{equation}\label{GfunaacDef}
		G^{\alpha}_{\beta,\eta,\sigma}(a,a,c;t)=\int_{0}^{t}(t-s)^{\beta+\eta-1}E_{\alpha,\beta+\eta}^{2}(a(t-s)^\alpha)s^{\sigma-1}E_{\alpha,\sigma}(cs^\alpha)\mathrm{d}s,
	\end{equation}
	where we have used \eqref{HfuncForm}. On taking the Laplace transform of \eqref{GfunaacDef} and by using \eqref{LTMittag}, we get
	\begin{align*}
		\mathbb{L}(G^{\alpha}_{\beta,\eta,\sigma}(a,a,c;t))(z)&=\frac{z^{3\alpha-\beta-\eta-\sigma}}{(z^\alpha-a)^2(z^\alpha-c)}\nonumber\\
		&=\frac{1}{(a-c)^2}\Big(\frac{a(a-c)z^{2\alpha-(\beta+\eta+\sigma)}}{(z^\alpha-a)^2}-\frac{acz^{\alpha-(\beta+\eta+\sigma)}}{z^\alpha-a}+\frac{c^2z^{\alpha-(\beta+\eta+\sigma)}}{z^\alpha-c}\Big),
	\end{align*}
	whose inverse Laplace transform yields
	\begin{equation*}
		G^{\alpha}_{\beta,\eta,\sigma}(a,a,c;t)=\frac{t^{\beta+\eta+\sigma-1}}{(a-c)^2}\big(a(a-c)E_{\alpha,\beta+\eta+\sigma}^{2}(at^\alpha)-acE_{\alpha,\beta+\eta+\sigma}(at^\alpha)+c^{2}E_{\alpha,\beta+\eta+\sigma}(ct^\alpha)\big).
	\end{equation*}
	In this way, we can obtain $G^{\alpha}_{\beta,\eta,\sigma}(a,b,c;t)$ for other cases too. It is given by
	{\smaller\begin{align*}\label{GFuncform}
		 G^{\alpha}_{\beta,\eta,\sigma}(a,b,c;t)=
		{\small\begin{cases}
			\vspace{0.2cm}
			t^{d-1}E_{\alpha,d}^{3}(at^\alpha), \, a=b=c,\\ \vspace{0.2cm}
			\frac{t^{d-1}}{(a-c)^2}\big(a(a-c)E_{\alpha,d}^{2}(at^\alpha)-acE_{\alpha,d}(at^\alpha)+c^{2}E_{\alpha,d}(ct^\alpha)\big), \, a=b\neq c,\\ \vspace{0.2cm}
			\frac{t^{d-1}}{a-b}\Big(\frac{a}{a-b}\big(aE_{\alpha,d}(at^\alpha)-bE_{\alpha,d}(bt^\alpha)\big)-bE_{\alpha,d}^{2}(bt^\alpha)\Big),\, a\neq b=c,\\ \vspace{0.2cm}
			\frac{t^{d-1}}{a-b}\Big(aE_{\alpha,d}^{2}(at^\alpha)-\frac{b}{a-b}\big(aE_{\alpha,d}(at^\alpha)-bE_{\alpha,d}(bt^\alpha)\big)\Big),\, c=a\neq b,\\ 
			\frac{t^{d-1}}{a-b}\Big(\frac{a^2}{a-c}E_{\alpha,d}(at^\alpha)-\frac{b^2}{b-c}E_{\alpha,d}(bt^\alpha)+\frac{c^2(a-b)}{(a-c)(b-c)}E_{\alpha,d}(ct^\alpha)\Big),\, a\neq b\neq c,
		\end{cases}}
	\end{align*}}
	where $d=\beta+\eta+\sigma$.
	
	
	
	\begin{theorem}\label{CondExpectVt}
		Conditional on the event that initial velocity is $v_1$, the expected velocity is given by
		\begin{equation*}
			\mathbb{E}(V_{\alpha}(t)|V_{\alpha}(0)=v_1)=v_1 E_{\alpha}(-\lambda t^\alpha)+v_1p\sum_{n=1}^{\infty}\psi_n(t|v_1)-v_2(1-p)\sum_{n=1}^{\infty}\xi_n(t|v_1),
		\end{equation*}
		where 
		{\small\begin{align*}
			\psi_n(t|v_1)
			&=\lambda\sum_{k=0}^{n-1}\binom{n-1}{k}(k+1)p^{k}(1-p)^{n-k-1}\sum_{i=0}^{n-k-1}\sum_{j=0}^{k}(-1)^{i+j}\binom{n-k-1}{i}\binom{k}{j}\nonumber\\
			&\hspace{1.5cm} \cdot\big(H^{\alpha}_{\alpha,1}(-\lambda(j+1),-\mu i;t)-\lambda(n+1)G^{\alpha}_{\alpha,1,\alpha}(-\lambda(j+1),-\mu i,-\lambda(n+1);t)\big)
		\end{align*}}
		and
		{\small\begin{align*}
			\xi_n(t|v_1)
			&=\lambda\sum_{k=0}^{n-1}\binom{n-1}{k}(k+1)p^{k}(1-p)^{n-k-1}\sum_{i=0}^{n-k-1}\sum_{j=0}^{k}(-1)^{i+j}\binom{n-k-1}{i}\binom{k}{j}\nonumber\\
			&\hspace{1.5cm} \cdot\big(H^{\alpha}_{\alpha,1}(-\lambda(j+1),-\mu i;t)-\mu(n+1)G^{\alpha}_{\alpha,1,\alpha}(-\lambda(j+1),-\mu i,-\mu(n+1);t)\big).
		\end{align*}}
	\end{theorem}
	\begin{proof}
		Consider the following integral:
		\begin{align}\label{IntPsi}
			\int_{0}^{t}
			&F_{T_n|S_{0}}(t-s|v_1)f_{R_{n+1}}(s)\mathrm{d}s\nonumber\\
			&=\lambda^2(n+1)\sum_{k=0}^{n-1}\binom{n-1}{k}(k+1)p^{k}(1-p)^{n-k-1}\sum_{i=0}^{n-k-1}\sum_{j=0}^{k}(-1)^{i+j}\binom{n-k-1}{i}\binom{k}{j}\nonumber\\
			&\ \ \cdot\int_{0}^{t} H^{\alpha}_{\alpha,1}(-\lambda(j+1),-\mu i;t-s)s^{\alpha-1}E_{\alpha,\alpha}(-\lambda(n+1)s^\alpha)\mathrm{d}s\nonumber\\
			&=\lambda^2(n+1)\sum_{k=0}^{n-1}\binom{n-1}{k}(k+1)p^{k}(1-p)^{n-k-1}\nonumber\\
			&\hspace{1cm} \cdot \sum_{i=0}^{n-k-1}\sum_{j=0}^{k}(-1)^{i+j}\binom{n-k-1}{i}\binom{k}{j} G^{\alpha}_{\alpha,1,\alpha}(-\lambda(j+1),-\mu i,-\lambda(n+1);t),
		\end{align}
		where we have used \eqref{GfuncDef} in the last step. Similarly, we get
		\begin{align}\label{IntXi}
			\int_{0}^{t}&F_{T_n|S_{0}}(t-s|v_1)f_{L_{n+1}}(s)\mathrm{d}s\nonumber\\
			&=\lambda\mu(n+1)\sum_{k=0}^{n-1}\binom{n-1}{k}(k+1)p^{k}(1-p)^{n-k-1}\nonumber\\
			&\hspace{1cm} \cdot \sum_{i=0}^{n-k-1}\sum_{j=0}^{k} (-1)^{i+j}\binom{n-k-1}{i}\binom{k}{j} G^{\alpha}_{\alpha,1,\alpha}(-\lambda(j+1),-\mu i,-\mu(n+1);t).
		\end{align}
		By using \eqref{IntPsi} and \eqref{IntXi} in Proposition 2 of Crimaldi {\it et al.} (2013), the required result is obtained. 
	\end{proof}
	
	\begin{remark}
		For $\alpha=1$, we have
		\begin{equation*}
			H^{1}_{1,1}(a,b;t)=\mathds{1}_{\{a=b\}}t e^{bt}+\mathds{1}_{\{a\neq b\}}\Big(\frac{e^{at}-e^{b t}}{a-b}\Big)
		\end{equation*}
		and
		\begin{align*}
			G^{1}_{1,1,1}(a,b,c;t)
			&=
			\begin{cases}
				\frac{1}{(a-c)^2}\big(e^{ct}-e^{at}+(a-c)te^{at}\big), \, a=b\neq c,\\
				\frac{e^{at}}{(b-a)(c-a)}+\frac{e^{bt}}{(a-b)(c-b)}+\frac{e^{ct}}{(b-c)(a-c)}, \, a\neq b\neq c.
			\end{cases}
		\end{align*}
		So, for $\alpha=1$, Theorem \ref{CondExpectVt} reduces to the corresponding result of generalized telegraph process with velocity driven by random trials: {\it Bernoulli scheme} (see Crimaldi {\it et al.} (2013), Proposition 4).
	\end{remark}
	
	\section{Telegraph process with generalized Mittag-Leffler waiting times and velocity driven by random trials}\label{SecGMLWaitTime}
	
	From Section \ref{SecStochModel}, we recall the generalized telegraph process with velocity driven by random trials: Case II, that is, $c>0$. Here, we analyze the {\it P\'olya urn scheme} with the assumption that $R_1$ and $L_1$ have the generalized Mittag-Leffler distributions, defined in \eqref{GenMLDist}, whereas $R_n$ and $L_n$ follow the Mittag-Leffler distribution for all $n \geq 2$. That is, we take the pdfs of $R_n$ and $L_n$ in the following form:
    \begin{align*}
    	f_{R_1}(t)&=\lambda^{\frac{b}{c}+1}t^{\alpha\big(\frac{b}{c}+1\big)-1}E_{\alpha,\alpha\big(\frac{b}{c}+1\big)}^{\frac{b}{c}+1}(-\lambda t^{\alpha}),\\
    	f_{L_1}(t)&=\mu^{\frac{r}{c}+1}t^{\alpha\big(\frac{r}{c}+1\big)-1}E_{\alpha,\alpha\big(\frac{r}{c}+1\big)}^{\frac{r}{c}+1}(-\mu t^{\alpha}),
    \end{align*}
    and for $n\geq2$, $f_{R_n}(t)=\lambda t^{\alpha-1}E_{\alpha,\alpha}(-\lambda t^\alpha)$, 
    $f_{L_n}(t)=\mu t^{\alpha-1}E_{\alpha,\alpha}(-\mu t^\alpha)$ for all $t\geq0$.
	
	The following result will be used.
	
	\begin{proposition}\label{PropX+YDist}
		Let $X$ be a generalized Mittag-Leffler random variable with parameters $\alpha>0$, $\delta>0$ and $\lambda>0$. Also, let $Y=Y_{1}+Y_{2}+\dots+Y_{n}$, where $Y_{i}$'s and $X$ are mutually independent such that $\bar{F}_{Y_i}(t)=E_{\alpha}(-\lambda t^\alpha)$ for all $1\leq i\leq n$. Then, $X+Y$ has generalized Mittag-Leffler distribution with parameters $\alpha$, $\delta+n$ and $\lambda$.
	\end{proposition}
	\begin{proof}
		The Laplace transform of $X+Y$ is given by
		\begin{equation}\label{LTX+Y}
			\mathbb{E}(e^{-z(X+Y)})
			=\mathbb{E}(e^{-zX})\prod_{i=1}^{n}\mathbb{E}(e^{-zY_{i}})
			=\Big(\frac{\lambda}{z^{\alpha}+\lambda}\Big)^{\delta+n}.
		\end{equation}
		On taking the inverse Laplace transform of \eqref{LTX+Y} and by using \eqref{LTMittag}, we get
		\begin{equation*}
			f_{X+Y}(t)=\lambda^{\delta+n}t^{\alpha(\delta+n)-1}E_{\alpha,\alpha(\delta+n)}^{\delta+n}(-\lambda t^\alpha),\,t\geq0.
		\end{equation*}
		This proves the result.
	\end{proof}
	
	Recall that $R^{(n)}=R_1 +R_2 +\cdots + R_n$ and $L^{(n)}=L_1 +L_2 +\cdots + L_n$, $n\in\mathbb{N}$. By using Proposition \ref{PropX+YDist}, their pdfs are given by 
	\begin{equation}\label{cpdfRn}
		f_{R}^{(n)}(t)=\lambda^{\frac{b}{c}+n}t^{\alpha\big(\frac{b}{c}+n\big)-1}E_{\alpha,\alpha\big(\frac{b}{c}+n\big)}^{\frac{b}{c}+n}(-\lambda t^{\alpha})
	\end{equation}
	and
	\begin{equation}\label{cpdfLn}
		f_{L}^{(n)}(t)=\mu^{\frac{r}{c}+n}t^{\alpha\big(\frac{r}{c}+n\big)-1}E_{\alpha,\alpha\big(\frac{r}{c}+n\big)}^{\frac{r}{c}+n}(-\mu t^{\alpha}),\,t\geq0,
	\end{equation}
	respectively.
	By using Eq. (5.1.19) of Gorenflo {\it et al.} (2014), their cdf are obtained in the following form:
	\begin{equation}\label{ccdfRn}
		F_{R}^{(n)}(t)=(\lambda t^\alpha)^{\frac{b}{c}+n}E_{\alpha,\alpha\big(\frac{b}{c}+n\big)+1}^{\frac{b}{c}+n}(-\lambda t^\alpha)
	\end{equation}
	and
	\begin{equation}\label{ccdfLn}
		F_{L}^{(n)}(t)=(\mu t^\alpha)^{\frac{r}{c}+n}E_{\alpha,\alpha\big(\frac{r}{c}+n\big)+1}^{\frac{r}{c}+n}(-\mu t^\alpha),\,t\geq0,
	\end{equation}
	respectively.
	
	We call this process as the telegraph process with generalized Mittag-Leffler waiting times and velocity driven by random trials, and denote it by $\{(\mathbb{X}_{\alpha}(t),\mathbb{V}_\alpha(t))\}_{t\geq0}$.
	
	For $\alpha=1$, \eqref{cpdfRn} and \eqref{cpdfLn} reduces to
	$f_{R}^{(n)}(t)={(\lambda^{{b}/{c}+n}t^{{b}/{c}+n-1}e^{-\lambda t})}/{\Gamma\big({b}/{c}+n\big)}$
	and 
	$f_{L}^{(n)}(t)={(\mu^{{r}/{c}+n}t^{{r}/{c}+n-1}e^{-\mu t})}/{\Gamma\big({r}/{c}+n\big)}$,
	respectively. That is, process  $\{(\mathbb{X}_{\alpha}(t),\mathbb{V}_{\alpha}(t))\}_{t\geq0}$ reduces to the generalized telegraph process with velocity driven by random trials: {\it P\'olya urn scheme} (see Crimaldi {\it et al.} (2013), p. 1126).
	
	
	Now, we obtain the discrete component of the probability law of $\{\mathbb{X}_{\alpha}(t)\}_{t\geq0}$, $t\geq0$ in terms of the Fox-Wright function defined in \eqref{DefFoxWright}.
	
	\begin{proposition}\label{cDiscProbLaw}
		The probabilities that the particle is at position $v_1 t$ and $-v_2 t$ are given by
		{\begin{align}\label{cDiscProbv1t}
			\mathrm{Pr}\{\mathbb{X}_{\alpha}(t)=v_1 t\}&=\frac{b}{b+r}\Big(1-(\lambda t^\alpha)^{\frac{b}{c}+1}E_{\alpha,\alpha(\frac{b}{c}+1)}^{\frac{b}{c}+1}(-\lambda t^\alpha)\Big)\nonumber\\
			&\ \ +\frac{b(\lambda t^\alpha)^{\frac{b}{c}+1}}{(b+r)\Gamma(\frac{b}{c}+1)}{}_2\Psi_2\!\left[
			\begin{array}{c}
				(\frac{b}{c}+2,1) (\frac{b+r}{c}+1,1)\vspace{0.5em}\\
				(\frac{b+r}{c}+2,1) (\alpha(\frac{b}{c}+1)+1,\alpha)
			\end{array}
			\,\middle|\, -\lambda t^\alpha
			\right]
		\end{align}}
	and
	\begin{align}\label{cDiscProbv2t}
		\mathrm{Pr}\{\mathbb{X}_{\alpha}(t)=-v_2 t\}&=\frac{r}{b+r}\Big(1-(\mu t^\alpha)^{\frac{r}{c}+1}E_{\alpha,\alpha(\frac{r}{c}+1)}^{\frac{r}{c}+1}(-\mu t^\alpha)\Big)\nonumber\\
		&\ \ +\frac{r(\mu t^\alpha)^{\frac{r}{c}+1}}{(b+r)\Gamma(\frac{r}{c}+1)}{}_2\Psi_2\!\left[
		\begin{array}{c}
			(\frac{r}{c}+2,1) (\frac{b+r}{c}+1,1)\vspace{0.5em}\\
			(\frac{b+r}{c}+2,1) (\alpha(\frac{r}{c}+1)+1,\alpha)
		\end{array}
		\,\middle|\, -\mu t^\alpha
		\right],
	\end{align}
	respectively.
	\end{proposition}
	\begin{proof}
		From \eqref{ccdfRn}, we have
		\begin{align}\label{cRncdfdiff}
			F_{R}^{(n)}(t)-F_{R}^{(n+1)}(t)
			&=(\lambda t^\alpha)^{\frac{b}{c}+n}\Big(E_{\alpha,\alpha\big(\frac{b}{c}+n\big)+1}^{\frac{b}{c}+n}(-\lambda t^\alpha)-\lambda t^\alpha E_{\alpha,\alpha\big(\frac{b}{c}+n+1\big)+1}^{\frac{b}{c}+n+1}(-\lambda t^\alpha)\Big)\nonumber\\
			&=(\lambda t^\alpha)^{\frac{b}{c}+n}E_{\alpha,\alpha\big(\frac{b}{c}+n\big)+1}^{\frac{b}{c}+n+1}(-\lambda t^\alpha),
		\end{align}
		where we have used the following relation (see Gorenflo {\it et al.} (2014), Eq. (5.1.12)):
		\begin{equation*}
			E_{a,ab+1}^{b}(-xy^a)-xy^aE_{a,a(b+1)+1}^{b+1}(-xy^a)=E_{a,ab+1}^{b+1}(-xy^a).
		\end{equation*}
		By using \eqref{cRncdfdiff} in Proposition 1 of Crimaldi {\it et al.} (2013), we get
		\begin{equation}\label{cProbXv1tCompo}
			\mathrm{Pr}\{\mathbb{X}_{\alpha}(t)=v_1 t\}
			=\frac{b}{b+r}\Big(1-(\lambda t^\alpha)^{\frac{b}{c}+1}E_{\alpha,\alpha\big(\frac{b}{c}+1\big)+1}^{\frac{b}{c}+1}(-\lambda t^\alpha)+h(t)\Big),
		\end{equation}
		where
		\begin{align}\label{cXv1tSeries}
			h(t)&=\sum_{n=1}^{\infty}
			\Big(\frac{b+c}{c}\Big)_n\Big(\Big(\frac{b+r+c}{c}\Big)_n\Big)^{-1}(\lambda t^\alpha)^{\frac{b}{c}+n}E_{\alpha,\alpha\big(\frac{b}{c}+n\big)+1}^{\frac{b}{c}+n+1}(-\lambda t^\alpha)\nonumber\\[0.5em]
			&=\sum_{n=1}^{\infty}\sum_{k=0}^{\infty}\Big(\frac{b+c}{c}\Big)_n\Big(\Big(\frac{b+r+c}{c}\Big)_n\Big)^{-1}\frac{(-1)^k(\lambda t^\alpha)^{n+k}\Gamma(\frac{b}{c}+n+1+k)(\lambda t^\alpha)^{\frac{b}{c}}}{k!\Gamma(\frac{b}{c}+n+1)\Gamma(\alpha(\frac{b}{c}+n+k)+1)}\nonumber\\[0.5em]
			&\hspace{10.7cm}(\text{by using \eqref{Mittag12}})\nonumber\\[0.5em]
			&=\frac{(\lambda t^\alpha)^{\frac{b}{c}}}{\Gamma(\frac{b}{c}+1)}\sum_{n=1}^{\infty}\sum_{k=0}^{\infty}\frac{(-1)^k\Gamma\big(\frac{b+r+c}{c}\big)\Gamma\big(\frac{b}{c}+n+k+1\big)(\lambda t^\alpha)^{n+k}}{k!\Gamma\big(\frac{b+r+c}{c}+n\big)\Gamma\big(\alpha(\frac{b}{c}+n+k)+1\big)}\nonumber \\[0.5em]  
			&=\frac{(\lambda t^\alpha)^{\frac{b}{c}}}{\Gamma(\frac{b}{c}+1)}\sum_{m=1}^{\infty}\frac{\Gamma\big(\frac{b+r+c}{c}\big)\Gamma\big(\frac{b+c}{c}+m\big)(\lambda t^\alpha)^m}{\Gamma\big(\alpha\big(\frac{b}{c}+m\big)+1\big)}\sum_{n=1}^{m}\frac{(-1)^{m-n}}{(m-n)!\Gamma\big(\frac{b+r+c}{c}+n\big)}\nonumber\\[0.5em]
			&=\frac{(\lambda t^\alpha)^{\frac{b}{c}}}{\Gamma(\frac{b}{c}+1)}\sum_{m=1}^{\infty}\frac{\Gamma\big(\frac{b+r+c}{c}\big)\Gamma\big(\frac{b+c}{c}+m\big)(\lambda t^\alpha)^m}{\Gamma\big(\alpha\big(\frac{b}{c}+m\big)+1\big)}\sum_{l=0}^{m-1}\frac{(-1)^{l}}{l!\Gamma\big(\frac{b+r+c}{c}+m-l\big)}\nonumber\\[0.5em]
			&=\frac{(\lambda t^\alpha)^{\frac{b}{c}}}{\Gamma(\frac{b}{c}+1)}\sum_{m=1}^{\infty}\frac{\Gamma\big(\frac{b+r+c}{c}\big)\Gamma\big(\frac{b+c}{c}+m\big)(\lambda t^\alpha)^m}{\Gamma\big(\alpha\big(\frac{b}{c}+m\big)+1\big)\Gamma\big(\frac{b+r+c}{c}+m\big)}\sum_{l=0}^{m-1}\frac{\Big(-\Big(\frac{b+r}{c}+m\Big)\Big)_l}{l!}\nonumber\\[0.5em]
			&\hspace{6cm}(\text{by using $(1-x)_{k}\Gamma(x-k)=(-1)^k\Gamma(x)$})\nonumber\\[0.5em]
			&=\frac{(\lambda t^\alpha)^{\frac{b}{c}}}{\Gamma(\frac{b}{c}+1)}\sum_{m=1}^{\infty}\frac{\Gamma\big(\frac{b+r+c}{c}\big)\Gamma\big(\frac{b+c}{c}+m\big)(\lambda t^\alpha)^m\big(1-\big(\frac{b+r}{c}+m\big)\big)_{m-1}}{(m-1)!\Gamma\big(\alpha\big(\frac{b}{c}+m\big)+1\big)\Gamma\big(\frac{b+r+c}{c}+m\big)}\nonumber\\[0.5em]
			&\hspace{8cm}\Big(\text{by using $\sum_{i=0}^{k}\frac{(a)_{i}}{i!}=\frac{(a+1)_{k}}{k!}$}\Big)\nonumber\\[0.5em]
			&=\frac{(\lambda t^\alpha)^{\frac{b}{c}}}{\Gamma\big(\frac{b+c}{c}\big)}\sum_{m=1}^{\infty}\frac{\Gamma\big(\frac{b+c}{c}+m\big)(\lambda t^\alpha)^m(-1)^{m-1}}{(m-1)!\big(\frac{b+r}{c}+m\big)\Gamma\big(\alpha\big(\frac{b}{c}+m\big)+1\big)}\nonumber\\[0.5em]
			&\hspace{6.4cm}(\text{by using $(1-x)_{k}\Gamma(x-k)=(-1)^k\Gamma(x)$})\nonumber\\[0.5em]
			&=\frac{(\lambda t^\alpha)^{\frac{b}{c}+1}}{\Gamma\big(\frac{b+c}{c}\big)}\sum_{m=0}^{\infty}\frac{\Gamma\big(\frac{b}{c}+m+2\big)(-\lambda t^\alpha)^m}{m!\big(\frac{b+r}{c}+m+1\big)\Gamma\big(\alpha\big(\frac{b}{c}+m+1\big)+1\big)}\nonumber\\[0.5em]
			&=\frac{(\lambda t^\alpha)^{\frac{b+c}{c}}}{\Gamma\big(\frac{b+c}{c}\big)}{}_2\Psi_2\!\left[
			\begin{array}{c}
				(\frac{b}{c}+2,1) (\frac{b+r+c}{c},1)\vspace{0.5em}\\ 
				(\frac{b+r+2c}{c},1) (\alpha(\frac{b+c}{c})+1,\alpha)
			\end{array}
			\,\middle|\, -\lambda t^\alpha
			\right],
		\end{align}
		where we have used \eqref{DefFoxWright} in the last step. On substituting \eqref{cXv1tSeries} in \eqref{cProbXv1tCompo}, we obtain \eqref{cDiscProbv1t}. Similarly, \eqref{cDiscProbv2t} can be established.
	\end{proof}
	
	\begin{remark}
		For $\alpha=1$, \eqref{cDiscProbv1t} reduces to
		\begin{align*}
			\mathrm{Pr}\{X(t)=v_1 t\}
			&=\frac{b}{b+r}\Big(1-(\lambda t)^{\frac{b}{c}+1}E_{1,\frac{b}{c}+2}^{\frac{b}{c}+1}(-\lambda t)\Big)\nonumber\\
			&\ \ +\frac{b(\lambda t)^{\frac{b}{c}+1}}{(b+r)\Gamma(\frac{b}{c}+1)}{}_2\Psi_2\!\left[
			\begin{array}{c}
				(\frac{b}{c}+2,1) (\frac{b+r}{c}+1,1)\vspace{0.5em}\\
				(\frac{b+r}{c}+2,1) (\frac{b}{c}+2,1)
			\end{array}
			\,\middle|\, -\lambda t
			\right]\\
			&=\frac{b}{b+r}\bar{F}_{R_1}(t)+\frac{b(\lambda t)^{\frac{b}{c}+1}}{(b+r)\Gamma(\frac{b}{c}+1)}\sum_{k=0}^{\infty}\frac{(-\lambda t)^k}{\big(\frac{b+r+c}{c}+k\big)k!}\\
			&=\frac{b}{b+r}\bar{F}_{R_1}(t)+\frac{bc(\lambda t)^{\frac{b}{c}+1}}{(b+r)\Gamma(\frac{b}{c}+1)(b+r+c)}\sum_{k=0}^{\infty}\frac{\big(\frac{b+r+c}{c}\big)_{k}(-\lambda t)^k}{\big(\frac{b+r+c}{c}+1\big)_{k}k!}\\
			&=\frac{b}{b+r}\bar{F}_{R_1}(t)+\frac{bc(\lambda t)^{\frac{b}{c}+1}e^{-\lambda t}}{(b+r)\Gamma(\frac{b}{c}+1)(b+r+c)}{}_{1}F_{1}(1,\frac{b+r+c}{c}+1;\lambda t)\\
			&=\frac{b}{b+r}\bar{F}_{R_1}(t)+\frac{b(\lambda t)^{\frac{b}{c}}e^{-\lambda t}}{(b+r)\Gamma(\frac{b}{c}+1)}\Big({}_{1}F_{1}\Big(1,\frac{b+r+c}{c};\lambda t\Big)-1\Big),
		\end{align*}
		which agrees with Eq. (35) of Crimaldi {\it et al.} (2013). Here, ${}_{1}F_{1}\big(1,\frac{b+r+c}{c};\lambda t\big)$ is Kummer's function. 
		So, for $\alpha=1$, the discrete component of probability of law of $\mathbb{X}_{\alpha}(t)$ reduces to that of the generalized telegraph process with velocity driven by random trials: {\it P\'olya urn scheme}.
	\end{remark}
	
	For the subsequent results, we use the same notations for the associated statistical quantities for $\{\mathbb{X}_{\alpha}(t)\}_{t\geq0}$ as used for $\{{X}_{\alpha}(t)\}_{t\geq0}$ in the previous section, for example, forward and backward densities, conditional distribution $n$th event time, {\it etc.}
	
	\begin{theorem}
		Conditional on the event that the initial velocity is $v_1$, the absolutely continuous component of probability law of $\{\mathbb{X}_{\alpha}(t)\}_{t\geq0}$  can be expressed in terms of its forward and backward densities as follows:
		{\small\begin{align}\label{cfxt}
			f(x,t|v_1)
			&=\frac{1}{v_1 +v_2}\sum_{n=2}^{\infty}\sum_{j=0}^{n-2}\binom{n-1}{j}\Big(\frac{b+c}{c}\Big)_{j+1}\Big(\frac{r}{c}\Big)_{n-j-1}\Big(\Big(\frac{b+r+c}{c}\Big)_n\Big)^{-1}\frac{(\mu(t-\theta)^\alpha)^{\frac{r}{c}+n-j-1}}{t-\theta}\nonumber\\
			&\hspace{2.7cm} \cdot (\lambda\theta^\alpha)^{\frac{b}{c}+j+1} E_{\alpha,\alpha\big(\frac{r}{c}+n-j-1\big)}^{\frac{r}{c}+n-j-1}(-\mu(t-\theta)^\alpha)E_{\alpha,\alpha\big(\frac{b}{c}+j+1\big)+1}^{\frac{b}{c}+j+2}(-\lambda\theta^\alpha),
		\end{align}}
	and
	{\small\begin{align}\label{cbxt}
		b(x,t|v_1)
		&=\frac{1}{v_1 +v_2}\bigg(\sum_{n=1}^{\infty}\Big(\frac{b+c}{c}\Big)_{n-1}\Big(\Big(\frac{b+r+c}{c}\Big)_n\Big)^{-1}\frac{r(\lambda\theta^\alpha)^{\frac{b}{c}+n}}{c\theta}E_{\alpha,\alpha\big(\frac{b}{c}+n\big)}^{\frac{b}{c}+n}(-\lambda\theta^\alpha)\bar{F}_{L_1}(t-\theta)\nonumber\\
		&\ \ +\sum_{n=2}^{\infty}\sum_{j=0}^{n-2}\binom{n-1}{j}\Big(\frac{b+c}{c}\Big)_{j}\Big(\frac{r}{c}\Big)_{n-j}\Big(\Big(\frac{b+r+c}{c}\Big)_n\Big)^{-1}\frac{(\lambda\theta^\alpha)^{\frac{b}{c}+j+1}}{\theta}(\mu(t-\theta)^\alpha)^{\frac{r}{c}+n-j-1}\nonumber\\
		&\hspace{1.7cm}\cdot  E_{\alpha,\alpha\big(\frac{b}{c}+j+1\big)}^{\frac{b}{c}+j+1}(-\lambda\theta^\alpha) E_{\alpha,\alpha\big(\frac{r}{c}+n-j-1\big)+1}^{\frac{r}{c}+n-j}(-\mu(t-\theta)^\alpha)\bigg),\, -v_2 t< x <v_1t
	\end{align}}
    respectively.
	\end{theorem}
	\begin{proof}
		Consider the following integral:
		{\smaller\begin{align}\label{cIntf}
			\int_{t-\theta}^{t}f_{R}^{(j+1)}(s-t+\theta)\bar{F}_{R_{j+2}}(t-s)\mathrm{d}s
			&=\lambda^{\frac{b}{c}+j+1}\int_{0}^{\theta}y^{\alpha\big(\frac{b}{c}+j+1\big)-1}E_{\alpha,\alpha\big(\frac{b}{c}+j+1\big)}^{\frac{b}{c}+j+1}(-\lambda y^{\alpha})E_{\alpha}(-\lambda(\theta-y)^\alpha)\mathrm{d}y\nonumber\\
			&=(\lambda\theta^\alpha)^{\frac{b}{c}+j+1}E_{\alpha,\alpha\big(\frac{b}{c}+j+1\big)+1}^{\frac{b}{c}+j+2}(-\lambda\theta^\alpha).
		\end{align}}
	Similarly,
	{\small\begin{equation}\label{cIntb}
			\int_{\theta}^{t}f_{L}^{(n-k-1)}(s-\theta)\bar{F}_{L_{n-k}}(t-s)\mathrm{d}s
			=(\mu(t-\theta)^\alpha)^{\frac{r}{c}+n-k-1}E_{\alpha,\alpha\big(\frac{r}{c}+n-k-1\big)+1}^{\frac{r}{c}+n-k}(-\mu(t-\theta)^\alpha).
	\end{equation}}
	From Crimaldi {\it et al.} (2013), we have
	{\small\begin{equation}\label{ProbMn-1,k|v1}
		\mathrm{Pr}\{M_{n-1}=k, S_n=v_1|S_0=v_1\}=\binom{n-1}{k}\Big(\frac{b+c}{c}\Big)_{k+1}\Big(\frac{r}{c}\Big)_{n-k-1}\Big(\Big(\frac{b+r+c}{c}\Big)_n\Big)^{-1}
	\end{equation}}
    and
    \begin{equation}\label{ProbMn-1,k-v2|v1}
    	\mathrm{Pr}\{M_{n-1}=k, S_n=-v_2|S_0=v_1\}=\binom{n-1}{k}\Big(\frac{b+c}{c}\Big)_{k}\Big(\frac{r}{c}\Big)_{n-k}\Big(\Big(\frac{b+r+c}{c}\Big)_n\Big)^{-1}.
    \end{equation}
	By using \eqref{cIntf} and \eqref{ProbMn-1,k|v1} in Eq. (15) of Crimaldi {\it et al.} (2013), we obtain
	{\smaller\begin{align}\label{cfnxt}
		f_{n}(x,t|v_1)
		&=\frac{1}{v_1 +v_2}\sum_{j=0}^{n-2}\binom{n-1}{j}\Big(\frac{b+c}{c}\Big)_{j+1}\Big(\frac{r}{c}\Big)_{n-j-1}\Big(\Big(\frac{b+r+c}{c}\Big)_n\Big)^{-1}\mu^{\frac{r}{c}+n-j-1}\nonumber\\
		&\ \ \cdot (t-\theta)^{\alpha\big(\frac{r}{c}+n-j-1\big)-1}E_{\alpha,\alpha\big(\frac{r}{c}+n-j-1\big)}^{\frac{r}{c}+n-j-1}(-\mu(t-\theta)^\alpha)(\lambda\theta^\alpha)^{\frac{b}{c}+j+1}E_{\alpha,\alpha\big(\frac{b}{c}+j+1\big)+1}^{\frac{b}{c}+j+2}(-\lambda\theta^\alpha).
	\end{align}}
	On both sides of \eqref{cfnxt}, we sum for all $n\geq2$ to get \eqref{cfxt}. 
    
    Now, by using \eqref{cIntb} and \eqref{ProbMn-1,k-v2|v1} in Eq. (16) of Crimaldi {\it et al.} (2013), we obtain
    {\small\begin{align}
    	b_{1}(x,t|v_{1})&=\frac{1}{v_1 +v_2}\frac{r(\lambda\theta^\alpha)^{\frac{b}{c}+1}}{(b+r+c)\theta}E_{\alpha,\alpha\big(\frac{b}{c}+1\big)}^{\frac{b}{c}+1}(-\lambda\theta^\alpha)\bar{F}_{L_1}(t-\theta),\label{cb1xt}\\
    	b_{n}(x,t|v_1)&=\frac{1}{v_1 +v_2}\Big(\Big(\frac{b+c}{c}\Big)_{n-1}\Big(\Big(\frac{b+r+c}{c}\Big)_n\Big)^{-1}\frac{r(\lambda\theta^\alpha)^{\frac{b}{c}+n}}{c\theta}E_{\alpha,\alpha\big(\frac{b}{c}+n\big)}^{\frac{b}{c}+n}(-\lambda\theta^\alpha)\bar{F}_{L_1}(t-\theta)\nonumber\\
    	&\ \ +\sum_{j=0}^{n-2}\binom{n-1}{j}\Big(\frac{b+c}{c}\Big)_{j}\Big(\frac{r}{c}\Big)_{n-j}\Big(\Big(\frac{b+r+c}{c}\Big)_n\Big)^{-1}\frac{(\lambda\theta^\alpha)^{\frac{b}{c}+j+1}}{\theta}E_{\alpha,\alpha\big(\frac{b}{c}+j+1\big)}^{\frac{b}{c}+j+1}(-\lambda\theta^\alpha)\nonumber\\
    	&\ \ \cdot(\mu(t-\theta)^\alpha)^{\frac{r}{c}+n-j-1}E_{\alpha,\alpha\big(\frac{r}{c}+n-j-1\big)+1}^{\frac{r}{c}+n-j}(-\mu(t-\theta)^\alpha)\Big),\, n\geq2.\label{cbnxt}
    \end{align}}
    On summing \eqref{cb1xt} and \eqref{cbnxt} for all $n\geq2$, we get \eqref{cbxt}.
    This proves the result.
	\end{proof}
	
	\begin{remark}
		On substituting $\alpha=1$ in \eqref{cfxt}, we obtain
		{\footnotesize\begin{align*}
			f(x,t|v_1)
			&=\frac{1}{v_1 +v_2}\sum_{n=2}^{\infty}\sum_{j=0}^{n-2}\binom{n-1}{j}\Big(\frac{b+c}{c}\Big)_{j+1}\Big(\frac{r}{c}\Big)_{n-j-1}\Big(\Big(\frac{b+r+c}{c}\Big)_n\Big)^{-1}\nonumber\\
			&\ \ \cdot \frac{(\mu(t-\theta))^{\frac{r}{c}+n-j-1}}{t-\theta} E_{1,\frac{r}{c}+n-j-1}^{\frac{r}{c}+n-j-1}(-\mu(t-\theta))(\lambda\theta)^{\frac{b}{c}+j+1}E_{1,\frac{b}{c}+j+2}^{\frac{b}{c}+j+2}(-\lambda\theta),\\
			&=\frac{1}{v_1 +v_2}\sum_{n=2}^{\infty}\sum_{j=0}^{n-2}\binom{n-1}{j}\Big(\frac{b+c}{c}\Big)_{j+1}\Big(\frac{r}{c}\Big)_{n-j-1}\Big(\Big(\frac{b+r+c}{c}\Big)_n\Big)^{-1}\\
			&\ \ \cdot \frac{(\mu(t-\theta))^{\frac{r}{c}+n-j-1}}{t-\theta} \frac{e^{-\lambda\theta-\mu(t-\theta)}}{\Gamma\big(\frac{r}{c}+n-j-1\big)\Gamma\big(\frac{b}{c}+j+2\big)}\\
			&=\frac{e^{-\lambda\theta-\mu(t-\theta)}(\lambda\theta)^{\frac{b}{c}+1}(\mu(t-\theta))^{\frac{r}{c}}}{(v_1 +v_2) (t-\theta)\Gamma\big(\frac{r}{c}\big)\Gamma\big(\frac{b}{c}+1\big)}\sum_{n=2}^{\infty}\Big(\Big(\frac{b+r+c}{c}\Big)_n\Big)^{-1}\sum_{j=0}^{n-2}\binom{n-1}{j}(\lambda\theta)^{j}(\mu(t-\theta))^{n-j-1}
		\end{align*}}
		which agrees with the corresponding result for generalized telegraph process with velocity driven by random trials: {\it P\'olya urn scheme} (see Crimaldi {\it et al.}(2013), p. 1128). On substituting $\alpha=1$ in \eqref{cbxt}, a similar reduction holds for $b(x,t|v_{1})$.
	\end{remark}
	
	\begin{remark}
			Conditional on the event that the initial velocity is $-v_2$, the absolutely continuous component of probability law of $\{\mathbb{X}_{\alpha}(t)\}_{t\geq0}$  can be expressed in terms of its forward and backward densities as follows:
			{\smaller\begin{align*}
					f(x,t|-v_2)
					&=\frac{1}{v_1 +v_2}\bigg( \bar{F}_{R_1}(\theta)\sum_{n=1}^{\infty}\Big(\frac{r+c}{c}\Big)_{n-1}\Big(\Big(\frac{b+r+c}{c}\Big)_n\Big)^{-1}\frac{b(\mu(t-\theta)^\alpha)^{\frac{r}{c}+n}}{c(t-\theta)}\nonumber\\
					&\ \ \cdot E_{\alpha,\alpha\big(\frac{r}{c}+n\big)}^{\frac{r}{c}+n}(-\mu(t-\theta)^\alpha) +\sum_{n=2}^{\infty}\sum_{j=0}^{n-2}\binom{n-1}{j}\Big(\frac{r+c}{c}\Big)_{j}\Big(\frac{b}{c}\Big)_{n-j}\Big(\Big(\frac{b+r+c}{c}\Big)_n\Big)^{-1}\nonumber\\
					&\ \ \cdot \frac{(\mu(t-\theta)^\alpha)^{\frac{r}{c}+j+1}}{t-\theta} (\lambda\theta^\alpha)^{\frac{b}{c}+n-j-1} E_{\alpha,\alpha\big(\frac{r}{c}+j+1\big)}^{\frac{r}{c}+j+1}(-\mu(t-\theta)^\alpha) E_{\alpha,\alpha\big(\frac{b}{c}+n-j-1\big)+1}^{\frac{b}{c}+n-j}(-\lambda\theta^\alpha)\bigg)
			\end{align*}}
		and
		{\smaller\begin{align*}
				b(x,t|-v_2)
				&=\frac{1}{v_1 +v_2}\sum_{n=2}^{\infty}\sum_{j=0}^{n-2}\binom{n-1}{j}\Big(\frac{r+c}{c}\Big)_{j+1}\Big(\frac{b}{c}\Big)_{n-j-1}\Big(\Big(\frac{b+r+c}{c}\Big)_n\Big)^{-1}\frac{(\lambda \theta^\alpha)^{\frac{b}{c}+n-j-1}}{\theta}\nonumber\\
				&\ \ \cdot (\mu(t-\theta)^\alpha)^{\frac{r}{c}+j+1} E_{\alpha,\alpha\big(\frac{b}{c}+n-j-1\big)}^{\frac{b}{c}+n-j-1}(-\lambda\theta^\alpha)E_{\alpha,\alpha\big(\frac{r}{c}+j+1\big)+1}^{\frac{r}{c}+j+2}(-\mu(t-\theta)^\alpha),\, -v_2 t< x <v_1t
		\end{align*}}
			respectively. If 	$\mathrm{Pr}\{\mathbb{V}_{\alpha}(0)=v_1\}=\mathrm{Pr}\{\mathbb{V}_{\alpha}(0)=-v_2\}=0.5$ then
			\begin{equation}\label{pxt1/2Polya}
				p(x,t)=\frac{1}{2}(p(x,t|v_1)+p(x,t|-v_2)).
			\end{equation}	
		\end{remark}

	\begin{figure}
		\centering
		\includegraphics[width=1\textwidth]{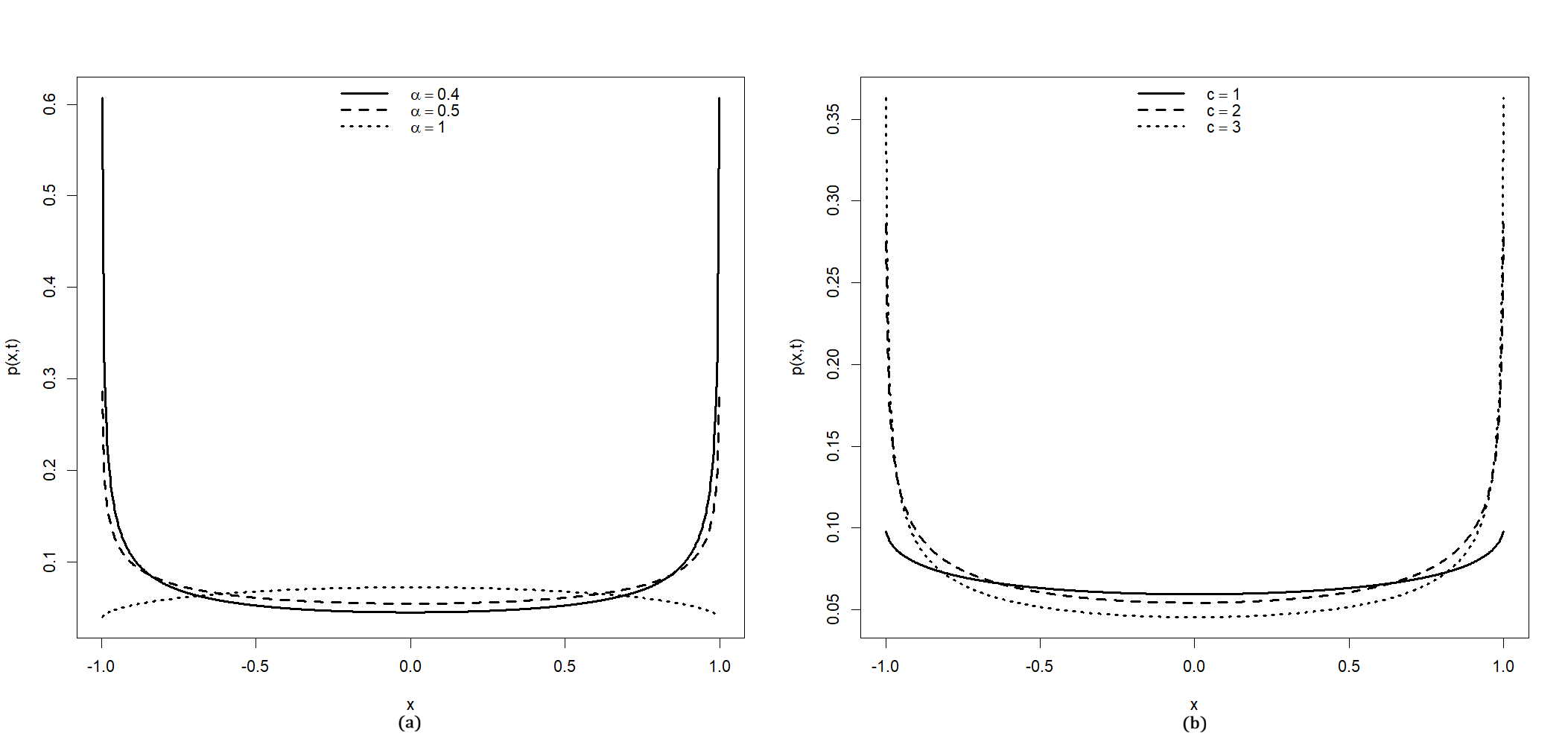}
		\caption{Plots of $p(x,t)$ (a) for different values of $\alpha$ with $t=1$, $v_1=v_2=1$, $\lambda=\mu=1$, $b=r=1$ and $c=2$, (b) for different values of $c$ with $t=1$, $\lambda=\mu=1$, $v_1=v_2=1$, $b=r=1$ and $\alpha=0.5$.}
		\label{fig4} 
	\end{figure}
	\begin{figure}
		\centering
		\includegraphics[width=1\textwidth]{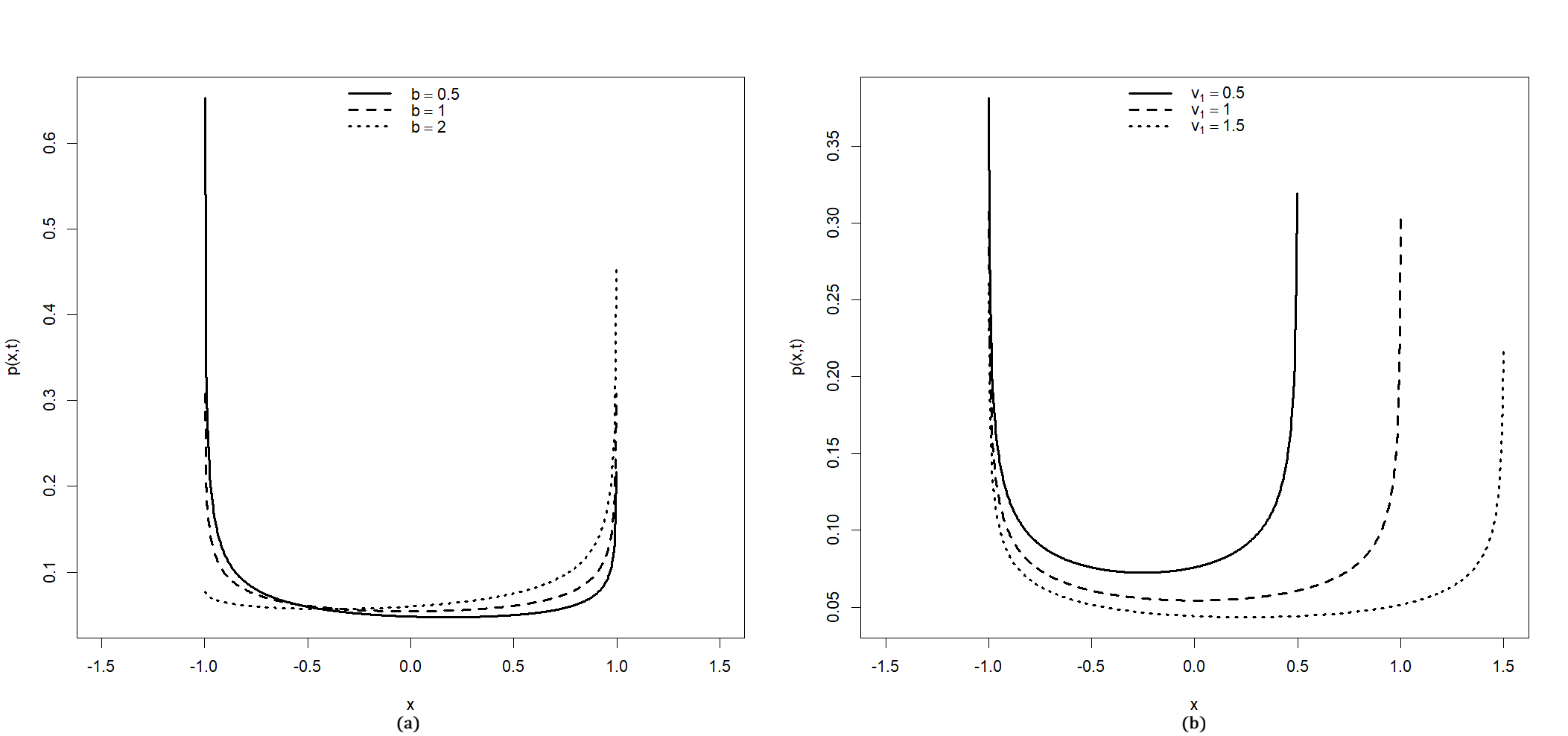}
		\caption{Plots of $p(x,t)$ (a) for different values of $b$ with $t=1$, $v_1=v_2=1$, $\lambda=\mu=1$, $r=1$, $c=2$ and $\alpha=0.5$, (b) for different values of $v_1$ with $t=1$, $\lambda=\mu=1$, $v_2=1$, $b=r=1$, $c=2$ and $\alpha=0.5$.}
		\label{fig5} 
	\end{figure}
	
	In Figure \ref{fig4}(a), the plots of $p(x,t)$ given in \eqref{pxt1/2Polya} are compared for different values of $\alpha$. It is observed that as $\alpha$ decreases, the plot becomes accumulated near boundaries of support of $\mathbb{X}_{\alpha}(t)$. For $\alpha=1$, it shows similar behavior to that observed in Figure 6 of Crimaldi {\it et al.} (2013). In Figure \ref{fig4}(b), the plots of $p(x,t)$ given in \eqref{pxt1/2Polya} are compared for different values of $c$. It is observed that as $c$ increases, plots gets accumulated near boundaries. In Figure \ref{fig5}(a), the plots of $p(x,t)$ given in \eqref{pxt1/2Polya} are compared for different values of $b$. As $b$ increases, the particle is more likely to move with velocity $v_1$ and so the plot becomes more concentrated towards the right. In Figure \ref{fig5}(b), the plots of $p(x,t)$ given in \eqref{pxt1/2Polya} are compared for different values of $v_1$. As $v_1$ increases, the support of $\mathbb{X}_{\alpha}(t)$ increases to the right.
	
	Consider the following convolution:
	\begin{align}\label{PrabhaConvo}
		t^{\beta_1-1}E_{\alpha_1,\beta_1}^{\gamma_1}(at^{\alpha_1})&*t^{\beta_2-1}E_{\alpha_2,\beta_2}^{\gamma_2}(bt^{\alpha_2})\nonumber\\
		&=\int_{0}^{t}s^{\beta_1-1}(t-s)^{\beta_2-1}E_{\alpha_1,\beta_1}^{\gamma_1}(as^{\alpha_1})E_{\alpha_2,\beta_2}^{\gamma_2}(b(t-s)^{\alpha_2})\mathrm{d}s\nonumber\\
		&=\sum_{k=0}^{\infty}\sum_{l=0}^{\infty}\frac{(\gamma_1)_k (\gamma_2)_l a^k b^l}{\Gamma(\alpha_1 k+\beta_1)\Gamma(\alpha_2 l+\beta_2)k!l!}\int_{0}^{t}s^{\alpha_1 k+\beta_1 -1}(t-s)^{\alpha_2 l+\beta_2-1}\mathrm{d}s\nonumber\\
		&=t^{\beta_1 + \beta_2 -1}\sum_{k=0}^{\infty}\sum_{l=0}^{\infty}\frac{(\gamma_1)_k (\gamma_2)_l (at^{\alpha_1})^k (bt^{\alpha_2})^l}{\Gamma(\alpha_1 k+\alpha_2 l+\beta_1+\beta_2)k!l!}\nonumber\\
		&=t^{\beta_1 + \beta_2 -1}E_{(\alpha_1,\alpha_2),\beta_1+\beta_2}^{(\gamma_1,\gamma_2)}(at^{\alpha_1},bt^{\alpha_2}).
	\end{align}
	Here,  $E_{(\alpha_1,\alpha_2),\beta}^{(\gamma_1,\gamma_2)}(\cdot,\cdot)$ is the bivariate generalized Mittag-Leffler function as defined in \eqref{MultGenML}. 

	 \begin{lemma}
	 	Conditional on the event that the initial velocity is $v_1$, the distribution of $n$th event time of counting process $\{\mathcal{N}(t)\}_{t\geq0}$
	 	associated with $\{\mathbb{X}_{\alpha}(t)\}_{t\geq0}$ is given by
	 	\begin{align*}
	 		F_{T_n |S_0}(t|v_1)&=\sum_{k=0}^{n-1}\binom{n-1}{k}\Big(\frac{b+c}{c}\Big)_k \Big(\frac{r}{c}\Big)_{n-k-1}\Big(\Big(\frac{b+r+c}{c}\Big)_{n-1}\Big)^{-1}\nonumber\\
	 		&\hspace{3cm} \cdot(\lambda t^\alpha)^{\frac{b}{c}+k+1}(\mu t^\alpha)^{\frac{r}{c}+n-k-1} E_{(\alpha,\alpha),\alpha\big(\frac{b+r}{c}+n\big)+1}^{(\frac{b}{c}+k+1,\frac{r}{c}+n-k-1)}(-\lambda t^\alpha,-\mu t^\alpha).
	 	\end{align*}
	 \end{lemma}
	 \begin{proof}
	 By using \eqref{cpdfRn} and \eqref{ccdfLn}, we get
	 	\begin{align}\label{cIntTnS0}
	 		\int_{0}^{t}&F_{L}^{(n-k-1)}(t-s)f_{R}^{(k+1)}(s)\mathrm{d}s\nonumber\\
	 		&=\lambda^{\frac{b}{c}+k+1}\mu^{\frac{r}{c}+n-k-1}\int_{0}^{t}(t-s)^{\alpha\big(\frac{r}{c}+n-k-1\big)}E_{\alpha,\alpha\big(\frac{r}{c}+n-k-1\big)+1}^{\frac{r}{c}+n-k-1}(-\mu(t-s)^\alpha)\nonumber\\
	 		&\hspace{6cm} \cdot s^{\alpha\big(\frac{b}{c}+k+1\big)-1} E_{\alpha,\alpha\big(\frac{b}{c}+k+1\big)}^{\frac{b}{c}+k+1}(-\lambda s^\alpha)\mathrm{d}s\nonumber\\
	 		&=(\lambda t^\alpha)^{\frac{b}{c}+k+1}(\mu t^\alpha)^{\frac{r}{c}+n-k-1}E_{(\alpha,\alpha),\alpha\big(\frac{b+r}{c}+n\big)+1}^{(\frac{b}{c}+k+1,\frac{r}{c}+n-k-1)}(-\lambda t^\alpha,-\mu t^\alpha),
	 	\end{align}
	 	where we have used \eqref{PrabhaConvo}. From Eq. (49) of Crimaldi {\it et al.} (2013), we have
	 	\begin{equation}\label{cProbMn-1S0}
	 		\mathrm{Pr}\{M_{n-1}=k|S_0 =v_1\}=\binom{n-1}{k}\Big(\frac{b+c}{c}\Big)_k \Big(\frac{r}{c}\Big)_{n-k-1}\Big(\Big(\frac{b+r+c}{c}\Big)_{n-1}\Big)^{-1}.
	 	\end{equation} 
	 	By using \eqref{cIntTnS0} and \eqref{cProbMn-1S0} in \eqref{ConditionalDistTn}, we get the required result.
	 \end{proof}
	 
	 Next, we derive the conditional expected velocity associated with the telegraph process with generalized Mittag-Leffler waiting times and velocity driven by random trials. 
	 
	 \begin{theorem}
	 	Conditional on the initial velocity $v_1$, the expected velocity associated with $\{\mathbb{X}_{\alpha}(t)\}_{t\geq0}$  is
	 	\begin{equation*}
	 		\mathbb{E}(\mathbb{V}_{\alpha}(t)|\mathbb{V}_{\alpha}(0)=v_1)= v_1 \bar{F}_{R_1}(t)+ \frac{v_1 (b+c)}{b+r+c}\sum_{n=1}^{\infty}\psi_n(t|v_1)- \frac{v_2 r}{b+r+c}\sum_{n=1}^{\infty}\xi_n(t|v_1),
	 	\end{equation*}
	 	where
	 	{\small\begin{equation*}
	 		\psi_n (t|v_1)=\sum_{k=0}^{n-1}\mathrm{Pr}\{M_{n-1}=k|S_0 =v_1\}(\lambda t^\alpha)^{\frac{b}{c}+k+1}(\mu t^\alpha)^{\frac{r}{c}+n-k-1} E_{(\alpha,\alpha),\alpha\big(\frac{b+r}{c}+n\big)+1}^{(\frac{b}{c}+k+2,\frac{r}{c}+n-k-1)}(-\lambda t^\alpha,-\mu t^\alpha)
	 	\end{equation*}}
	 	and
	 	{\small\begin{equation*}
	 			\xi_n (t|v_1)=\sum_{k=0}^{n-1}\mathrm{Pr}\{M_{n-1}=k|S_0 =v_1\}(\lambda t^\alpha)^{\frac{b}{c}+k+1}(\mu t^\alpha)^{\frac{r}{c}+n-k-1} E_{(\alpha,\alpha),\alpha\big(\frac{b+r}{c}+n\big)+1}^{(\frac{r}{c}+n-k,\frac{b}{c}+k+1)}(-\mu t^\alpha,-\lambda t^\alpha).
	 	\end{equation*}}
	 	\begin{proof}
	 		By using \eqref{ccdfLn} and Proposition \ref{PropX+YDist}, we have
	 		\begin{align}
	 			\int_{0}^{t}F_{L}^{(n-k-1)}&(t-s)f_{R^{(k+1)}+R_{n+1}}(s)\mathrm{d}s\nonumber \\
	 			&=\mu^{\frac{r}{c}+n-k-1}\lambda^{\frac{b}{c}+k+2}\int_{0}^{t}(t-s)^{\alpha\big(\frac{r}{c}+n-k-1\big)}E_{\alpha,\alpha\big(\frac{r}{c}+n-k-1\big)+1}^{\frac{r}{c}+n-k-1}(-\mu(t-s)^\alpha)\nonumber\\
	 			&\hspace{6cm} \cdot y^{\alpha\big(\frac{b}{c}+k+2\big)-1}E_{\alpha,\alpha\big(\frac{b}{c}+k+2\big)}^{\frac{b}{c}+k+2}(-\lambda s^\alpha)\mathrm{d}s \label{cIntLRDEf}\\
	 			&=(\lambda t^\alpha)^{\frac{b}{c}+k+2}(\mu t^\alpha)^{\frac{r}{c}+n-k-1} E_{(\alpha,\alpha),\alpha\big(\frac{b+r}{c}+n+1\big)+1}^{(\frac{b}{c}+k+2,\frac{r}{c}+n-k-1)}(-\lambda t^\alpha,-\mu t^\alpha).\label{cIntLRn1}
	 	   \end{align}
	 		By using \eqref{cIntTnS0} and \eqref{cIntLRn1} in Remark 2 of Crimaldi {\it et al.} (2013), we get
	 	{	\small\begin{align*}
	 			\psi_n(t|v_1)
	 			&=\sum_{k=0}^{n-1}\mathrm{Pr}\{M_{n-1}=k|S_0 =v_1\} \int_{0}^{t}F_{L}^{(n-k-1)}(t-s)\big(f_{R}^{(k+1)}(s)-f_{R^{(k+1)}+R_{n+1}}(s)\big)\mathrm{d}s\\
	 			&=\sum_{k=0}^{n-1}\mathrm{Pr}\{M_{n-1}=k|S_0 =v_1\}(\lambda t^\alpha)^{\frac{b}{c}+k+1}(\mu t^\alpha)^{\frac{r}{c}+n-k-1}\nonumber\\
	 			&\hspace{1cm} \cdot  \Big(E_{(\alpha,\alpha),\alpha\big(\frac{b+r}{c}+n\big)+1}^{(\frac{b}{c}+k+1,\frac{r}{c}+n-k-1)}(-\lambda t^\alpha,-\mu t^\alpha)-\lambda t^{\alpha}E_{(\alpha,\alpha),\alpha\big(\frac{b+r}{c}+n+1\big)+1}^{(\frac{b}{c}+k+2,\frac{r}{c}+n-k-1)}(-\lambda t^\alpha,-\mu t^\alpha)\Big)\nonumber\\
	 			&=\sum_{k=0}^{n-1}\mathrm{Pr}\{M_{n-1}=k|S_0 =v_1\}(\lambda t^\alpha)^{\frac{b}{c}+k+1}(\mu t^\alpha)^{\frac{r}{c}+n-k-1} E_{(\alpha,\alpha),\alpha\big(\frac{b+r}{c}+n\big)+1}^{(\frac{b}{c}+k+2,\frac{r}{c}+n-k-1)}(-\lambda t^\alpha,-\mu t^\alpha),
	 		\end{align*}}
	 		where we have used the following relation in the last step:
	 		\begin{equation*}
	 			E_{(a,a),a(b+c)+1}^{(b,c)}(-xz^a,-yz^a)-xz^a E_{(a,a),a(b+c+1)+1}^{(b+1,c)}(-xz^a,-yz^a)=E_{(a,a),a(b+c)+1}^{(b+1,c)}(-xz^a,-yz^a).
	 		\end{equation*}
	 		Similarly, we can obtain $\xi_n(t|v_1)$. This proves the result.
	 	\end{proof}
	 \end{theorem}
	 
	 \begin{remark}
	 	On substituting $\alpha=1$ in \eqref{cIntLRDEf}, we have
	 	\begin{align*}
	 		\int_{0}^{t}&F_{L}^{(n-k-1)}(t-s)f_{R^{(k+1)}+R_{n+1}}(s)\mathrm{d}s\nonumber \\
	 		&=\lambda^{\frac{b}{c}+k+2}\mu^{\frac{r}{c}+n-k-1}\int_{0}^{t}(t-s)^{\frac{r}{c}+n-k-1}s^{\frac{b}{c}+k+1}E_{1,\frac{r}{c}+n-k}^{\frac{r}{c}+n-k-1}(-\mu(t-s))E_{1,\frac{b}{c}+k+2}^{\frac{b}{c}+k+2}(-\lambda s)\mathrm{d}s\nonumber\\
	 		&=\frac{\lambda^{\frac{b}{c}+k+2}\mu^{\frac{r}{c}+n-k-1}}{\Gamma\big(\frac{b}{c}+k+2\big)\Gamma\big(\frac{r}{c}+n-k\big)}\\
	 		&\hspace{3cm} \cdot \int_{0}^{t}(t-s)^{\frac{r}{c}+n-k-1}s^{\frac{b}{c}+k+1}e^{-\mu(t-s)}{}_{1}F_{1}\Big(1,\frac{r}{c}+n-k,\mu(t-s)\Big)e^{-\lambda s}\mathrm{d}s,
	 	\end{align*}
	 	where ${}_{1}F_{1}\big(1,{r}/{c}+n-k,\mu(t-s)\big)$ is the Kummer's function.
	 	So, conditional on the event that the initial velocity of the particle is $v_1$, the expected velocity of particle associated with $\{\mathbb{X}_{\alpha}(t)\}_{t\geq0}$ reduces to that of the generalized telegraph process with velocity driven by random trials: {\it P\'olya urn scheme} (see Appendix B of Crimaldi {\it et al.} (2013)).
	 \end{remark}
	 
	\section*{Acknowledgement}
	The research of first author was supported by a UGC fellowship, NTA reference no. 231610158041, Govt. of India.
	
\end{document}